\documentclass
[
a4paper,
DIV=11,
12pt,
abstracton
]
{scrartcl}

\usepackage{fullpage,authblk,amsmath,amssymb,amsthm, mathtools}
\usepackage{bm,bbm}
\usepackage{enumerate,paralist}
\usepackage{graphicx}
\usepackage[table]{xcolor}
\usepackage{tikz}
\usetikzlibrary{cd} 
\usetikzlibrary{arrows,decorations.pathmorphing,backgrounds,fit,positioning,shapes.symbols,chains}
\usetikzlibrary{arrows.meta} 
\usepackage{pgfplots}
\usepackage{multirow}
\usepackage[bf, format=plain]{caption} 
\usepackage{subcaption}
\usepackage{dsfont}
\usepackage{stackengine} 
\usepgfplotslibrary{fillbetween}
\usepackage{comment}

\usepackage[pdffitwindow=false,
plainpages=false,
pdfpagelabels=true,
pdfpagemode=UseOutlines,
pdfpagelayout=SinglePage,
bookmarks=false,
colorlinks=true,
hyperfootnotes=false,
linkcolor=blue,
urlcolor=blue!30!black,
citecolor=green!50!black]{hyperref}

\newtheorem{definition}{Definition}

\newtheorem{lemma}[definition]{Lemma}
\newtheorem{proposition}[definition]{Proposition}

\newtheorem{remark}[definition]{Remark}
\newtheorem{theorem}[definition]{Theorem}
\newtheorem{conj}[definition]{Conjecture}

\newtheorem*{TakThm}{Takens' theorem}
\newtheorem*{KTThm}{Koopman--Takens theorem}

\newcommand{\R}{\mathbb{R}}
\newcommand{\N}{\mathbb{N}}
\newcommand{\C}{\mathbb{C}}
\newcommand{\ep}{\varepsilon}
\renewcommand{\epsilon}{\varepsilon}

\newcommand{\set}[1]{\mathbb{#1}}
\newcommand{\Meng}[2]{\left\{#1\mathrel{}\middle|\mathrel{}#2\right\}}
\newcommand{\abs}[1]{\left\lvert#1\right\rvert}
\newcommand{\norm}[1]{\left\lVert#1\right\rVert}
\newcommand{\cA}{\mathcal{A}}

\newcommand{\cK}{\mathcal{K}}
\newcommand{\cU}{\mathcal{U}}
\newcommand{\bT}{\mathbb{T}}
\newcommand{\Gd}{{\text{G}$_{\delta}$}}

\newcommand{\mpt}{\mathrm{MPT}}
\newcommand{\refl}{\mathcal{R}}
\newcommand{\FT}{\mathcal{F}}
\newcommand{\homeo}{\mathrm{Homeo}}

\newcommand{\revision}[1]{{#1}}

\graphicspath{{./graphics/}}

\title{A Koopman--Takens theorem: Linear least squares prediction of nonlinear time series}
\author[1,2]{P\'eter~Koltai}
\author[3]{Philipp Kunde}

\affil[1]{Department of Mathematics, University of Bayreuth, Germany}
\affil[2]{Institute of Mathematics, Freie Universit\"at Berlin, Germany}
\affil[3]{Faculty of Mathematics and Computer Science, Jagiellonian University in Kraków, Poland}

\date{}

\begin{document}
	
	\maketitle
	
	\begin{abstract}
		The least squares linear filter, also called the Wiener filter, is a popular tool to predict the next element(s) of time series by linear combination of time-delayed observations. We consider observation sequences of deterministic dynamics, and ask: Which pairs of observation function and dynamics are predictable? If one allows for nonlinear mappings of time-delayed observations, then Takens' well-known theorem implies that a set of pairs, large in a specific topological sense, exists for which an exact prediction is possible. We show that a similar statement applies for the linear least squares filter in the infinite-delay limit, by considering the forecast problem for invertible measure-preserving maps and the Koopman operator on square-integrable functions.
		
		

	\end{abstract}


	\section{Introduction}
	
	The prediction (or forecast or extrapolation) of time series is important in diverse applications~\cite{HBS,HsiehEtAl,Fish,SugiharaEtAl,SugiharaMay} and attracts lively research activity~\cite{BGS20,BGS22,baranski2022prediction,FS87,KostelichYorke,SaYoCa91,packard1980geometry,Takens,Voss,WangGuet2022}.
	It can be framed as follows. Given a discrete-time dynamics $T:\set{X} \to \set{X}$ on state space $\set{X}$ and observation function $f:\set{X}\to\C$, are we able to assess from the time series $f(x), f(Tx), \ldots, f(T^n x)$ its next element~$f(T^{n+1}x)$?
	
	We are going to consider for this problem the linear least (mean) squares filter~\cite[section~9.7]{percival1993spectral}. Going back to Wiener~\cite{wiener1949extrapolation} and Kolmogorov~\cite{kolmogorov1941stationary,kolmogorov1941interpolation}, it is also called \emph{Wiener filter}. It computes linear combination coefficients for $d$ consecutive observations $f(T^ix), \ldots, f(T^{i+d-1}x)$ such that the square error of the linear combination to the next observations in the sequence, $f(T^{i+d}x)$, averaged over the sequence is minimal. The positive integer $d$ is called \emph{delay depth}.
	
	If we were to allow nonlinear transformations to map the time-delayed observations,
	Takens' celebrated theorem \cite{Takens} implies that an exact prediction is possible for a ``large'' set of pairs of map $T$ and observable~$f$:
	\begin{TakThm}
		Let $\set{M}$ be a compact manifold of dimension $m$. The set of pairs of mappings and observables $(T,f)\in \mathrm{Diff}^2(\set{M}) \times C^2(\set{M},\R)$, for which the delay-coordinate map 
		\[
		\Phi_{T,f}:\set{M} \to \R^{2m+1}, \ \Phi_{T,f}(x)=\left(f(x),f(Tx), \ldots , f(T^{2m}x)\right)
		\]
		is an embedding, is open and dense with respect to the $C^1$ topology.
	\end{TakThm}

	To be more precise, exact predictability is implied by the consequence of Takens' theorem that the mapping $\left(f(T^ix), \ldots, f(T^{i+d-1}x) \right) \mapsto \left(f(T^{i+1}x), \ldots, f(T^{i+d}x) \right)$ with $d=2m+1$ is one-to-one as the composition $\Phi_{T,f}(x) \mapsto x \mapsto Tx \mapsto \Phi_{T,f}(Tx)$ of (smooth) one-to-one mappings.
	Note that a part of this pipeline is the \emph{reconstruction} step $\Phi_{T,f}(x) \mapsto x$, possible by the injectivity of
	the map $\Phi_{T,f}$.
	Closer to our way of looking at the problem, \smash{\cite{baranski2022prediction}} considers the (less restrictive) problem of predicting from $d$ consecutive terms of the time series its future values. Notice that such \emph{prediction} takes place not in the original phase space $\set{M}$, but in $\Phi_{T,f}(\set{M})\subset \R^d$, considered as a model space for the system. See also~\cite{Takens2}. 
	
	The descriptor ``large'' from above refers to the adjectives `open and dense' from Takens' theorem---often called genericity in this context. It is customary to call a property generic if it holds for all points of a set that is large in a topological sense. This largeness, however, has no unique description throughout the literature. Indeed, in most cases the word ``generic'' refers to a property that holds for a set of points containing a dense \Gd{} set. Here, a set is called~{\Gd} if it is a countable intersection of open sets, cf.\ Definition~\ref{def:Gd} below. Note that this is slightly weaker than the genericity used in Takens' theorem itself.
	
	Takens' theorem has been extended over the years in different directions, such as less smoothness of the setting~\cite{SaYoCa91,BGS20,BGS22,Gutman,Robinson} (see also \cite{robinson_Book,baranski2022prediction} for a comprehensive overview) or stronger concepts for the largeness of the set for which it holds. One such concept is \emph{prevalence}~\cite{SaYoCa91}, which is a generalization of ``Lebesgue-almost every'' for infinite-dimensional normed spaces; cf.\ Definition~\ref{def:prevalence}.
	
	Takens-type theorems serve as a justification of the validity of time-delay based approaches used in applications, and have been the basis of further methodological developments, see e.g.~\cite{broomhead1986extracting,FS87,SugiharaMay,SugiharaEtAl,KostelichYorke,HsiehEtAl,dellnitz2016computation,brunton2017chaos,Fish,kamb2020time,gottwald2021combining}.
	
	At this point the prediction problem seems to reduce to learning the scalar function $\left(f(x), \ldots, f(T^{d-1}x) \right) \mapsto f(T^dx)$ for a sufficiently large delay depth. However, this function is nonlinear and the curse of dimensionality becomes a burden, as no additional exploitable structural property of it is known. Although, for instance, deep-learning based approaches have been deployed to this learning problem with success~\cite{lapedes1987nonlinear,young2023deep,bakarji2023discovering,WangGuet2022}, training such artificial neuronal networks remains a nontrivial optimization problem and performance guarantees are not available. 
	\revision{Similar difficulties arise for ``composite'' methods, where for a fixed set of nonlinear ansatz functions a linear combination optimizing the forecast error is sought~\cite{gottwald2021supervised,gottwald2021combining,wulkow2021data} and for the related family of regression methods based on the Mori--Zwanzig formalism~\cite{zwanzig1961statistical,Mori65}, utilizing a projection-based quantification of how the unobserved past propagates to the present and future~\cite{gouasmi2017apriori,lin2021data,gilani2021lernel,lin23regression}.}
	Due to its simplicity and (also theoretic) accessibility, the linear filter remains a viable tool for prediction. In addition, as we will see below, the forecast problem is all about approximating a linear operator.
	
	Indeed, we interpret the prediction problem by the linear least squares filter in the setting of ergodic theory, with an essentially invertible measure-preserving transformation $T: (\set{X},\mathfrak{B}, \mu) \to (\set{X},\mathfrak{B}, \mu)$ and its associated Koopman operator~$\cU :L^2_\mu \to L^2_\mu$ with~$\cU f = f\circ T$~\cite{koopman1931hamiltonian,koopman1932dynamical}, often also called composition operator.
	We will see in Proposition~\ref{prop:L2projU} below that the least squares filter in the limit of infinite data corresponds to an orthogonal projection of $\cU$ to the Krylov subspace~$\cK_d = \mathrm{span}\{ f\circ T^{-d+1}, \ldots, f\circ T^{-1}, f\} \subset L^2_{\mu}$ spanned by time-delayed observables. 
	
	Without an explicit mentioning of it, a first connection of the linear least squares filter and the Koopman operator in this projective sense seems to have appeared in~\cite{arbabi2017ergodic}. Therein, the focus was on the connection to a more recent linear method, Dynamic Mode Decomposition (DMD), which arose in the context of fluid dynamics~\cite{SS08}. The relationship of DMD and the Koopman operator was observed in~\cite{rowley2009spectral}, after which generalizations such as Extended DMD~\cite{williams2015data} appeared. This led to the reintroduction of the Wiener filter with the name ``Hankel DMD'' in~\cite{arbabi2017ergodic}.
	
	The marriage of Koopman operator and delay embedding has since then witnessed a lively activity. Intermittency in chaotic systems was investigated by the ``Hankel alternative view of Koopman analysis''~\cite{brunton2017chaos,kamb2020time}, spectral approximation results of Hankel DMD have been derived in~\cite{korda2020data}, while others have also sought for spectral approximation beyond the linear filter scenario~\cite{GiDa19,korda2020data,giannakis2021delay,colbrook2021rigorous,colbrook2023mpedmd,valva2023consistent}.
	
	Note that prediction requires the inference of $\cU f(x)$ from past observations \linebreak[4] $\ldots, f\circ T^{-2}(x), f\circ T^{-1}(x), f(x)$. The usual trade-off in Koopman-operator based methods is that the original nonlinear but (potentially) finite-dimensional system is considered through a linear, but infinite-dimensional operator. It is thus not surprising that in general perfect prediction by linear manipulations of time-delayed observations cannot be achieved, merely in the limit of infinite delay depth, i.e.\ ``asymptotically''.
	We will consider two notions of \emph{asymptotic linear predictiveness}. The first is (strong) predictiveness, requiring that the closure $\overline{\bigcup_{d\in \N}\cK_d} = L^2_{\mu}$; i.e., that the Krylov spaces ``fill''~$L^2_{\mu}$, cf.\ Definition~\ref{def:predictability}. The second, weak predictiveness, we define by~$\smash{ \cU f \in \overline{\bigcup_{d\in \N}\cK_d} }$, cf.\ Definition~\ref{def:weakpred}.
	With regard to the above discussion after Takens' theorem, we note that the stronger version of predictiveness here is closer to the notion of reconstruction (of the Koopman operator), while the weaker one is in line with what one considered above as prediction in model space. 
	
	In the following we consider the Wiener filter and our results will mainly concern (strong) linear asymptotic predictiveness. We provide some Takens-type predictiveness theorems where the prediction map is linear. Using the notions of standard probability space without atoms (that is, a measure space isomorphic to $[0,1]$ with Lebesgue measure) and weak topology on measure-preserving essential bijections from Section~\ref{subsec:basicsMT}, our main results are summarized in a simplified fashion as follows: 
	\begin{KTThm}
		Let $(\set{X},\mathcal{B},\mu)$ be a standard probability space without atoms and let $\mpt$ be the space of $\mu$-preserving essential bijections on $\set{X}$ with the weak topology.
		\begin{enumerate}[(a)]
			\item Theorem~\ref{thm:DensePrevalent}: For a dense set of transformations $T \in \mpt$, the collection of asymptotically linearly predictive observables~$f$ is prevalent and generic in $L^2_{\mu}$.
			\item Proposition~\ref{prop:DenseGd}: For a generic set of transformations $T \in \mpt$, the collection of asymptotically linearly predictive observables~$f$ is generic in $L^2_{\mu}$.
			\item Proposition~\ref{prop:DiscSpecWeakPredict}: For a dense set of transformations $T \in \mpt$, every $L^2_\mu$-observable $f$ is weakly linearly asymptotically predictive. Specifically, the conclusion holds for all transformations $T$ with discrete spectrum.
		\end{enumerate}
	\end{KTThm}
	Practically this means that for a large set of systems almost any measurement procedure (i.e., observable) results in a predictive scenario. However, it also means, that on finite delay depths the Wiener filters of predictive and of unpredictive systems can be arbitrary close (Proposition~\ref{prop:filterapprox}).
	A result of our analysis is also that there are mixing systems that are predictive (section~\ref{subsec:RankOne}). Since mixing systems ``lose their memory'' over time, it was not clear a priori whether observations from arbitrary far in the past can contribute enough for this to be the case.
	
	We can compare the Takens and Koopman--Takens theorems along different lines:\\
	(i) While for Takens' theorem the (often unknown) dimension of the manifold $\set{M}$ provides a bound on the minimal delay depth for predictability to hold, in the Koopman--Takens setting statements hold in the asymptotic limit of infinite delay depth. In fact, in this latter ergodic setting the structure of a manifold is not needed; we will only require Lebesgue--Rokhlin spaces below (section~\ref{subsec:basicsMT}). \revision{However, in section~\ref{subsec:homeo} we can extend our results to the setting of measure-preserving homeomorphisms on compact connected manifolds.}\\
	(ii) Takens-type results require the mapping and the observable to be at least (sometimes Lipschitz) continuous, while our framework works with measurable maps and observables.
	It should also be observed that the prediction problem can be considered for both invertible and non-invertible transformations~$T$, while the possibility of dynamical reconstruction (embedding) by delay coordinates maps is generally restricted to invertible systems in the Takens framework. It is important to note that from a theoretical point of view, reconstruction implies prediction but not vice versa. In the Koopman--Takens framework we only consider essentially invertible maps.\\
	(iii) For Takens' theorem, one can identify a class of $T\in\mathrm{Diff}^2(\set{M})$ such that the conclusion of the theorem holds: If $T$ has finitely many periodic points of period at most $2m$ and for any periodic point $x$ of period $p\leq 2m$ the eigenvalues of $\mathrm{D}T^p(x)$ are all distinct, then for generic $f \in C^2(\set{M},\R)$ the map $\Phi_{T,f}$ is an embedding~\cite[Theorem~2]{huke2006embedding}.
	In our setting, one can also give a class of systems for which weak predictiveness holds: Mappings with discrete spectrum are weakly predictive for \emph{every} $L^2_\mu$-observable (see Proposition~\ref{prop:DiscSpecWeakPredict}).
	
	The paper is structured as follows.
	Section~\ref{sec:filter} sets up the prediction problem in an ergodic-theoretic context and shows that an important role is played by so-called cyclic vectors of the Koopman operator. The size of the set of predictive pairs of dynamics and observables is considered in section~\ref{sec:genericity}, where we give different genericity results for predictiveness---both of constructive and nonconstructive nature. Further practical properties of the least squares linear filter for time series analysis are collected in section~\ref{sec:filter_properties}. We illustrate our findings on four numerical examples in section~\ref{sec:examples}, before we conclude with section~\ref{sec:outlook}.

	\section{The least-squares filter}
	\label{sec:filter}
	
	\subsection{Preliminaries on operator-based considerations in ergodic theory}
	
	The measure-theoretic consideration of dynamical systems starts with a measure space $(\set{X},\mathfrak{B}, \mu)$ and a measure-preserving transformation (mpt)~$T:\set{X}\to \set{X}$, i.e.,~$\mu = \mu\circ T^{-1}$. If the $\sigma$-algebra $\mathfrak{B}$ is clear from the context, we suppress it in the notation, and write~$(\set{X},\mu)$. We will work with probability spaces, thus~$\mu(\set{X})=1$.
	
	The system (or, if no ambiguity arises, the \emph{dynamics} or the measure) is called ergodic if there are only trivial invariant sets, that is, for every set $E\in \mathcal{B}$ with $\mu(T^{-1}(E) \triangle E)=0$ we have $\mu(E)=0$ or $\mu(E)=1$. Every system can be restricted to subsets $\set{X}' \subset \set{X}$, such that $T\big\vert_{\set{X}'}$ is ergodic; so-called ergodic components.
	For ergodic systems, the Birkhoff Ergodic Theorem states that for $f\in L^p_\mu := L^p_\mu(\set{X})$, $1\le p < \infty$, the orbital averages converge, more precisely
	\[
	\lim_{n\to \infty} \frac1n \sum_{i=0}^{n-1} f\left( T^i x\right) = \int_{\set{X}} f\,d\mu
	\]
	for $\mu$-a.e.\ $x\in\set{X}$ and also as a function in~$L^p_\mu$. Note that this is a functional version of the indecomposability statement defining ergodicity: The orbital averages of any observables converge to a constant. Again, if $T$ is not ergodic, this statement holds restricted to its ergodic components. 
	
	This functional view on ergodic theory (also) motivates to consider observables and their evolution under the dynamics, giving rise to the Koopman operator~$\cU: L^2_\mu\to L^2_\mu$, $\cU f = f\circ T$. The analogously defined operator could be considered on $L^p_\mu$, $1\le p\le \infty$, as well, but the additional structure of a Hilbert space and the possibility to consider autocorrelations leads to our choice of~$L^2_\mu$. Sometimes, to omit ambiguity, we write~$\cU_T$ to denote the Koopman operator associated to~$T$. Since $T$ preserves the measure, we have that $\smash{\| \cU f\|_{L^2_\mu} = \|f\|_{L^2_\mu} }$, and if $T$ is (essentially) invertible, then $\cU$ is \emph{unitary}. Then, in particular, $\cU_{T^{-1}} = \cU_T^{-1}$ and the spectrum of $\cU$ is contained in~$\{z\in\C \mid |z|=1\}$, since both $\cU$ and $\cU^{-1}$ are non-expansive.
	This spectrum can be further classified as follows. The mpt $T$ (or the operator $\cU$) is said to have (purely) \emph{discrete spectrum}, if $\cU$ has a countable set of eigenvalues with associated eigenfunctions yielding an orthonormal basis $\phi_i$, $i\in\N$, of~$L^2_\mu$. Note that $\lambda_1=1$ is always an eigenvalue of $\cU$ with eigenfunction~$\phi_1 = \mathds{1}$ (the constant function with value one). If 1 is the only eigenvalue of $\cU$, then $T$ (and $\cU$) is said to have (purely) \emph{continuous spectrum}. This is for instance the case for mixing systems, while rotations of abelian groups have purely discrete spectrum~\cite[\S3.3]{Wal00}. One speaks of $T$ having mixed spectrum if neither of the above pure cases hold. We note that the terminology is consistent with the functional-analytic partitioning of spectra into the point, continuous, and residual parts. Indeed, unitary operators have no residual spectrum: one way to see this known fact is to combine \cite[Theorem~9.2-4]{kreyszig1991introductory} and \cite[section~10.6]{kreyszig1991introductory}.
	
	If $f,g \in L^2_\mu$, then $f\, \cU^k g \in L^1_\mu$ for $k\ge 0$, and the Birkhoff Ergodic Theorem yields
	\begin{equation}
		\label{eq:tuples}
		\lim_{n\to \infty} \frac1n \sum_{i=0}^{n-1} f\left( T^i x\right) g\left( T^{i+k} x\right) = \lim_{n\to \infty} \frac1n \sum_{i=0}^{n-1} (f\, g\circ T^k)\left( T^i x\right) =  \int_{\set{X}} f\, \cU^k g \,d\mu
	\end{equation}
	for $\mu$-a.e.~$x\in \set{X}$.
	
	Finally, the form of the observation sequence suggests to consider sequences of the form~$(f,\cU f, \cU^2 f,\ldots)$.
	If the so-called \emph{cyclic subspace} of $f$, defined by 
	\[
	C_f := C_{f,\,\cU} := \overline{\mathrm{span} \{ f,\cU f, \cU^2 f,\ldots\}},
	\]
	is equal to~$L^2_\mu$, then we call $f$ a \emph{cyclic vector} of~$\cU$. As we will see, such cyclic observables play an important role for the prediction problem.

	\subsection{The filter as $L^2$-orthogonal projection}
	
	In all what follows, we assume $T$ to be a $\mu$-ergodic essentially invertible mpt on~$\set{X}$. Let a trajectory
	\[
	x,Tx,\ldots, T^mx
	\]
	of length $(m+1)$ be given. For a given delay depth $d\in\N$, let us consider the linear least squares filtering problem for forecasting the next observation from the previous $d$ ones:
	\begin{equation}
		\label{eq:filter_finite}
		\tilde{J}_m(\tilde{c}) := \frac{1}{m-d+1}\sum_{t=0}^{m-d} \abs{f(T^{t+d}x) - \sum_{i=0}^{d-1} \tilde{c}_i f(T^{t+i}x) }^2 \to \min_{\tilde{c}\in\C^{d}}!,
	\end{equation}
	where $\to\min_{\tilde{c}\in\C^{d}}!$ expresses the goal of minimizing $J(\tilde{c})$ over $\tilde{c}=(\tilde{c}_0,\dots,\tilde{c}_{d-1})\in\C^{d}$.
	For $\mu$-a.e.\ $x\in \set{X}$ the value function converges as $m\to\infty$ by~\eqref{eq:tuples} and we obtain the limiting filtering problem
	\begin{equation}
		\label{eq:filter}
		\tilde{J}(\tilde{c}) := \int \abs{ f(T^d x) - \sum_{i=0}^{d-1} \tilde{c}_i f(T^ix) }^2 \, d\mu(x) \to \min_{\tilde{c}\in\C^{d}}!
	\end{equation}
	Replacing $\tilde{J}_m$ by $\tilde{J}$ is reasonable in the following sense.
	\begin{lemma}
		\label{lem:filter_convergence}
		Let $d\in\N$ be fixed and let $\{f, f\circ T,\ldots,f\circ T^{d-1}\}$ be a linearly independent set in~$L^2_\mu$. Then the minimizer of $\tilde{J}_m$ converges for almost every~$x\in\set{X}$ to the unique minimizer of $\smash{\tilde{J}}$ for~$m\to\infty$.
	\end{lemma}
	\begin{proof}
		The linear independence assumption guarantees that $\tilde{J}:\C^d \to \R$ is a quadratic loss function with symmetric positive definite $d\times d$ Hessian matrix $H$. Thus, $\tilde{J}$ has a unique minimizer, that can be characterized as the unique solution of a linear system $H \tilde{c} = b$ with a suitable right-hand side~$b$. Also, $\tilde{J}_m:\C^d\to\R$ is a quadratic loss function, whose finitely many parameters converge by the Birhoff Ergodic Theorem for a.e.\ $x\in\set{X}$ to those of~$\tilde{J}$. Thus, $\tilde{J}_m$ has almost surely a unique minimizer for sufficiently large~$m$, which converges essentially because $(H,b) \mapsto H^{-1}b$ is continuous for invertible matrices~$H$.
	\end{proof}
	
	From now on, we will thus consider the infinite-data case. More precisely, we will work with the following equivalent formulation of problem~\eqref{eq:filter}. Note that by invariance of $\mu$ with respect to $T$ and hence with respect to $T^{d-1}$ as well, we have with $c_i := \tilde{c}_{d-1-i}$ that
	\begin{align}	
		\tilde{J}(\tilde{c}) &\stackrel{\phantom{j=d-1-i}}{=} \int \abs{ f(T x) - \sum_{i=0}^{d-1} \tilde{c}_i f(T^{i-d+1}x) }^2 \, d\mu(x) \nonumber\\
		&\stackrel{j=d-1-i}{=} \int \abs{ f(T x) - \sum_{j=0}^{d-1} c_j f(T^{-j} x) }^2 \, d\mu(x)	=: J(c). \label{eq:filter_new}
	\end{align}
	We note that at this point it is convenient to write the (finite) trajectory, by redefining its starting point, as
	\[
	T^{-m}x, \ldots, T^{-1}x,x,
	\]
	whenever we interpret the ergodic limit in the finite-data case.
	
	To interpret the minimizer of $J$ in~\eqref{eq:filter_new}, we define the Krylov ``basis'' $\mathcal{B}_d:= (f,  f\circ T^{-1}, \ldots, f\circ T^{-d+1})$ and the associated Krylov space $\mathcal{K}_d := \mathrm{span}\, \mathcal{B}_d$. The Krylov basis is a basis of $\mathcal{K}_d$ if and only if it forms a linearly independent set.
	
	\begin{remark}
		Note that if $f$ is a cyclic vector of $\cU^{-1}$, then $\mathcal{B}_d$ is a basis for any~$d\ge 1$, assuming that $L^2_\mu$ is infinite-dimensional.
	\end{remark}
	
	Since $J(c) \to \min!$ is the best approximation problem in $L^2_\mu$, the minimizer $c = (c_0,\ldots,c_{d-1})$ of $J$ induces a function
	\[
	\cU_d f:= \sum_{j=0}^{d-1} c_j f \circ T^{-j} \in \mathcal{K}_d
	\]
	that is the unique $L^2_{\mu}$-orthogonal projection of $\cU f$ onto~$\mathcal{K}_d$. There will be no confusion of notation between $\cU_T$ and $\cU_d$, as $T$ will always be a map and $d$ an integer.
	If $\mathcal{B}_d$ is a basis, the associated coefficient vector $c$ is unique. The orthogonal projection property is equivalent to the variational equation (by setting the gradient of $J$ to zero)
	\begin{equation}
		\label{eq:vareq}
		\langle \cU_d f - \cU f, f\circ T^{-i} \rangle_{L^2_\mu} = 0\qquad \forall i=0,\ldots,d-1.
	\end{equation}
	We can extend the operator $\cU_d$ to $\mathcal{K}_d$ by defining it to have the matrix representation
	\begin{equation}
		\label{eq:Ud}
		U_d = \begin{bmatrix}
			c_0 & 1 & & & \\
			c_1 & 0 & 1 & & \\
			\vdots & & \ddots & \ddots & \\
			c_{d-2} & & & 0 & 1 \\
			c_{d-1} & & & & 0
		\end{bmatrix}
	\end{equation}
	with respect to $\mathcal{B}_d$, where the first column is the vector $c$, the first upper diagonal contains ones, and all other entries are zero. By setting up the variational equation for the $L^2_\mu$-orthogonal projection of the Koopman operator $\cU$ on the space $\mathcal{K}_d$, it is immediate to see the following 
	\begin{proposition}
		\label{prop:L2projU}
		The operator $\cU_d: \mathcal{K}_d \to \mathcal{K}_d$ with matrix representation $U_d$ with respect to $\mathcal{B}_d$ is the $L^2_\mu$-orthogonal projection of $\cU$ onto~$\mathcal{K}_d$.
	\end{proposition}
	
	\begin{remark}[Hankel DMD]
		A statement similar to Proposition~\ref{prop:L2projU} has been made in~\cite[section~3]{arbabi2017ergodic} and was used there in the case~$\cK_{d+1} = \cK_d$ for some~$d\in\N$. Therein, the filter was coined \emph{Hankel Dynamic Mode Decomposition (Hankel DMD)}, because it arises as the least squares solution of the (usually overdetermined) linear system
		\begin{equation}
			\label{eq:Hankel}
			\begin{bmatrix}
				f(T^{-1}x) & \cdots & f(T^{-d}x)\\
				f(T^{-2}x) & \cdots & f(T^{-d-1}x) \\
				\vdots & & \vdots \\
				f(T^{-m+d-1}x) & \cdots & f(T^{-m}x)
			\end{bmatrix}
			\begin{bmatrix}
				c_0\\ \vdots \\ c_{d-1}
			\end{bmatrix}
			=
			\begin{bmatrix}
				f(x) \\
				f(T^{-1}x) \\
				\vdots \\
				f(T^{-m+d}x)
			\end{bmatrix}
		\end{equation}
		involving a Hankel-type matrix (more precisely, a rectangular Toeplitz matrix) in the construction of (Extended) Dynamic Mode Decomposition~\cite{SS08,williams2015data}.
	\end{remark}
	
	\begin{remark}[Computational aspects]
		\revision{There are several issues that need to be taken into account when computing the filter numerically from finite data.}
		
		\revision{The problem of obtaining the filter vector $c$ via \eqref{eq:filter_new} inherits the numerical condition of the Krylov basis~$\mathcal{B}_d$. If the basis elements are highly non-orthogonal, the computation of $c$ is ill-conditioned. This need not effect the performance of the filter, though. This is because the filter's performance is measured by the objective function~$J$ in $L^2_\mu$ and not in the space of coefficients~$c$.}
		
		\revision{A more severe source of error is the finite-sample approximation~\eqref{eq:filter_finite} of~$J$. From an ergodic-theoretic perspective, the convergence in Lemma~\ref{lem:filter_convergence} as $m\to \infty$ can be arbitrarily slow. From the perspective of quadrature, integrals of the form~\eqref{eq:vareq} involve functions $f\circ T^{-d+1}$, which might become increasingly irregular as $d$ grows and thus slowing down convergence. Such integrals enter implicitly while computing inner products in the solution of the least squares problem~\eqref{eq:Hankel}. A remedy (simple, but without performance guarantee) is to compute the filter for larger sample sizes $m$ (or start with reducing the set, if obtaining new samples is not an option) and thus checking whether convergence can be safely assumed.} 
	\end{remark}
	
	We now return to our original question, whether the filter is able to continue an observation sequence $\ldots,f(T^{-1}x),f(x)$. Hence, we would like to infer $\cU f(x)$. \revision{While there are also other mathematical frameworks for quantifying predictability and information content in ensemble predictions (see \cite{Bröcker,Cai,Gneiting} and references therein), we focus on the \emph{mean squared forecast error} and can give the following result.}
	\begin{proposition}
		\label{prop:forecasterror}
		Let $f$ be a cyclic vector for~$\cU^{-1}$. Then, as $d\to\infty$, we have
		\[
		\| \,\cU_d f - \cU f \|_{L^2_\mu} \to 0.
		\]
	\end{proposition}
	\begin{proof}
		Since $f$ is a cyclic vector, the space $\bigcup_{d\ge 0} \cK_d$ is dense in $L^2_\mu$. Since $\cU_d$ is the orthogonal projection of $\cU$ to $\cK_d$, and the spaces $\cK_d$ are nested, i.e., $\cK_{d_1} \subset \cK_{d_2}$ for $d_1\le d_2$, this proves that $\cU_d f$ converges to $\cU f$.
	\end{proof}
	
	Proposition~\ref{prop:forecasterror} motivates the following notion of predictiveness.
	\begin{definition}[Predictiveness]
		\label{def:predictability}
		We call the pair $(T,f)$ of a mpt and an $L^2_{\mu}$-observable \emph{asymptotically linearly predictive in the $L^2$-sense}, if $f$ is a cyclic vector of~$\cU^{-1} = \cU_{T^{-1}}$. If $T$ and $\mu$ are clear from the context, we simply say that $f$ is \emph{predictive}.
	\end{definition}
	
	\begin{remark}
		\revision{The term ``predictive'' is motivated by the observation that if $f \in L^2_\mu$ is a cyclic vector of $\cU^{-1}$, then for any $g \in L^2_{\mu}$ we can approximate $\cU g$ by a linear combination of $\cU^{-j}f$, $j\in \N$, to any desired accuracy in $L^2_{\mu}$ norm by minimizing a modified version of $\tilde{J}(\tilde{c})$ in \eqref{eq:filter}.}
	\end{remark}
	
	An interesting question is, whether $L^2_\mu$ convergence in Proposition~\ref{prop:forecasterror} can be replaced by almost sure pointwise convergence.
	
	\begin{remark}[Almost sure asymptotic predictiveness]
		Convergence of a sequence in $L^2_\mu$ implies convergence almost everywhere for a subsequence. Thus, if $f$ is a cyclic vector of~$\cU^{-1}$, there is a sequence $(d_k)_{k\in \N}$ of delay depths for which ``almost sure asymptotic linear predictiveness'' holds.
		However, we do not know in general which subsequence this is.
		In practice, an a posteriori test for convergence of a sequence $\big(\cU_{d_k}f(x) \big)_{k\in\N}$ could be our only option.
		
		We also note that almost-everywhere pointwise convergence of orthogonal projections on $L^2$ spaces is considered in \cite[Theorems~2 and~3]{Dun77}, but it seems unclear whether the conditions therein (e.g., positivity) on the projection can be met in our situation (i.e., $L^2_\mu$-orthogonal projection on~$\cK_d$). It seems rather unlikely for an arbitrary observation function~$f$.
	\end{remark}
	
	Proposition~\ref{prop:forecasterror} shows that cyclic vectors of $\cU^{-1}$ are of particular interest and importance to the forecast properties of the filter. In section~\ref{sec:genericity} we shall discuss the size of the set of cyclic vectors and the size of the set of transformations that have ``many'' of them. Before doing so, we consider a weaker but ``sufficient'' notion of predictiveness.
	
	\subsection{On weak predictiveness}
	\label{ssec:ST}
	
	\revision{
		Observe that in order to approximate the next element of the time series by a linear filter, it is only required\footnote{For a possible nonlinear analogue of this in the context of Takens' theorem, see~\cite[Definition~1.8]{baranski2022prediction}.} that $\cU f$ can be arbitrarily well approximated in $L^2_\mu$ by polynomials in $\cU^{-1}$ applied to~$f$. }
	\revision{
		\begin{definition}[Weak predictiveness]
			\label{def:weakpred}
			We call the pair $(T,f)$, or simply the observable $f$, \emph{weakly predictive} if $f \circ T \in \overline{\bigcup_{d\in\N} \cK_d}$.
			Equivalently, weak predictiveness means $\lim_{d\to\infty} d_{L^2_\mu}(\cU f, \cK_d) = 0$, where $d_{L^2_\mu}$ denotes the distance in $L^2_\mu$ of a function to a subspace.
		\end{definition}
		In the following, we are almost exclusively going to consider the stronger notion of predictiveness implied by the existence of cyclic vectors, as in Definition~\ref{def:predictability} and Proposition~\ref{prop:forecasterror}. It is nevertheless worthwhile discussing the weaker notion presented here, since observables that are eigenfunctions of $\cU$ (such as constant functions) will always be weakly but never strongly predictive (except for pathological cases where $L^2_\mu$ is one-dimensional). In particular, we give a characterization of weak predictiveness in Theorem~\ref{thm:weakpred} and show in Proposition~\ref{prop:DiscSpecWeakPredict} that if the system has discrete spectrum, then \emph{every} observable $f$ is weakly predictive.
	}
	
	To analyze weak predictiveness of a system, we will use the Spectral Theorem for unitary operators. Let $\bT \subset \C$ denote the unit circle.
	In a simplified form (namely by restricting the operator in question to a cyclic subspace), the Spectral Theorem of unitary operators states that $\cU$ is isomorphic to the multiplication operator on $L^2_\nu(\bT)$ for a suitable measure~$\nu$. The general form shows isomorphy to a sum of such spaces. See, for instance, \cite[Chapter~8]{taylor2013partial} or~\cite[Chapter~18]{EFHN15} for the general statement and also for the form discussed next in section~\ref{subsec:STsummary}.

	\subsubsection{Spectral Theorem for unitary operators}
	\label{subsec:STsummary}
	
	As a warm-up, consider a unitary (with respect to the Euclidean inner product~$\langle \cdot,\cdot \rangle$) matrix $A \in \C^{k\times k}$ with eigenvalues $\lambda_i \in \bT$, $i=1,\ldots,k$, with an orthonormal set of eigenvectors~$v_i$. We can write the image of a vector $v = \sum_i c_i v_i$ under $p(A)$, where $p$ is a polynomial, as~$p(A)v = \sum_i p(\lambda_i) c_i v_i$. Then we have that
	\[
	\langle p(A)v,v \rangle = \sum_i p(\lambda_i) |c_i|^2 = \int_{\bT} p(\lambda) \, \Big( \sum_i |c_i|^2\delta_{\lambda_i} \Big)(d\lambda),
	\]
	where $\delta_x$ denotes the Dirac measure centered at~$x$. The measure $\nu = \sum_i |c_i|^2\delta_{\lambda_i}$ thus encodes a weighted version of the spectrum of $A$ associated with the vector~$v$.
	
	The above finite-dimensional construction can be generalized to unitary operators on Hilbert spaces. We will not consider the most general case, it will be sufficient to restrict a unitary operator $\cA:H \to H$ on the Hilbert space $H$ to the cyclic subspace $C_{v,\cA}$ of $\cA$ and $v\in H$:
	\[
	C_{v,\cA} = \overline{ \mathrm{span} \left\{ v,\cA v,\cA^2 v,\ldots \right\} }.
	\]
	Note that $C_{v,\cA}$ is invariant under~$\cA$. By Riesz duality, one finds a unique ($v$-dependent) measure $\nu$ (which we will call the ``trace of the spectral measure for $v$'', or simply ``trace measure'') with
	\[
	\langle h(\cA)v,v \rangle = \int_{\bT} h\, d \nu \quad \forall h\in C(\bT).
	\]
	In fact, this equality can be extended to hold for every $h\in L^2_{\nu}(\bT)$, and one can define a unitary transformation
	\[
	W: L^2_{\nu}(\bT) \to H,\quad W h:= h(\cA)v.
	\]
	Note that if $\mathds{1}$ denotes the constant function with value 1 on~$\mathbb{T}$, then~$W\mathds{1} = v$. Ultimately, by applying $W$ to the function $\zeta \mapsto \zeta h(\zeta)$, one has that
	\begin{equation}
		\label{eq:ST}
		W^{-1}\cA Wh (\zeta) = \zeta h(\zeta)\quad \forall h\in L^2_{\nu}(\bT).    
	\end{equation}
	Thus, on its spectrum, $\cA$ can be characterized by the multiplication operator $M$, with $(Mh)(\zeta) := \zeta h(\zeta)$, just as for matrices. The Spectral Theorem comprises the statements that $W$ is a unitary transformation and that \eqref{eq:ST} holds.
	
	\subsubsection{Application of the Spectral Theorem}
	
	For weak predictiveness, we would need $\cU^{-1}f \in C_{f,\, \cU}$ (for simplicity of notation, we swapped the roles of $T$ and $T^{-1}$). In other words, for any $\ep>0$ we would like to have a polynomial $p_{\ep}$ such that
	\[
	\| \cU^{-1}f - p_{\ep}(\cU)f\|_{L^2_\mu} < \ep.
	\]
	Note that the Spectral Theorem can be used in its above simplified form on the invariant cyclic subspace~$C_{f,\,\cU}$ for $v = f$ and $\mathcal{A} = \mathcal{U}$. With this and $W\mathds{1} = f$, one has that
	\begin{align*}
		\| \cU^{-1}f - p_{\ep}(\cU)f\|_{L^2_\mu} < \ep \quad &\Leftrightarrow \quad \| WM^{-1}W^{-1}f - Wp_{\ep}(M)W^{-1} f\|_{L^2_\mu} < \ep \\
		&\Leftrightarrow \quad \| M^{-1}\mathds{1} - p_{\ep}(M)\mathds{1}\|_{L^2_\nu} < \ep.
	\end{align*}
	Spelling out the last expression, we would like to approximate the function $\zeta \mapsto \zeta^{-1} $ arbitrarily well by polynomials in the $L^2_\nu$ norm. Note that if $\nu = \mathrm{Leb}$, then this is impossible, because on the complex circle $\bT$ the monomials $\zeta \mapsto \zeta^k$, $k\in\mathbb{Z}$, are orthogonal trigonometric polynomials.
	If, however, $\nu$ is a finite sum of weighted Dirac measures (for finitely many distinct eigenvalues $\lambda_i$), then there is a polynomial $p$ with $p(\lambda_i) = \frac{1}{\lambda_i}$ for all~$i$. This follows from the invertibility of a Vandermonde matrix.
	A much broader answer, containing these previous two examples, can be given by invoking an implication of Szeg{\H o}'s theorem:
	
	\begin{theorem}[{Kolmogorov's Density Theorem~\cite[Theorem 2.11.5]{simon2010szegHo}}]
		\label{thm:KolmogorovDensity}
		Let $\nu$ be a probability measure on the complex circle $\bT$ of the form $d\nu(\theta) = w(\theta)\,d\theta + d\nu_s(\theta)$, where $\nu_s$ is singular and~$d\theta$ measures arc length. Then the polynomials are dense in $L^2_\nu(\bT)$ if and only if
		\[
		\int \log(w(\theta))\,d\theta = -\infty.
		\]
	\end{theorem}
	
	\revision{
		\begin{theorem}
			\label{thm:weakpred}
			Let $d\nu(\theta) = w(\theta)\,d\theta + d\nu_s(\theta)$ denote the decomposition of the trace measure $\nu$ associated with $T$ and $f$ as in Theorem~\ref{thm:KolmogorovDensity}. Then $f$ is weakly predictive if and only if the \emph{Szeg{\H o} condition} $\int \log(w(\theta))\,d\theta = -\infty$ holds.
		\end{theorem}
	}
	\begin{proof}
		\revision{We recall that---with abbreviating notation such that $\zeta^n$ denotes the rational function~$\zeta\mapsto \zeta^n$ for $n\in\mathbb{Z}$---weak predictiveness is equivalent to that $\zeta^{-1}$ can be arbitrary well approximated by monomials~$\mathds{1},\zeta,\zeta^2,\ldots$ in~$L^2_\nu$. Equivalently, we write $\zeta^{-1} \in \overline{\mathrm{span}\{\mathds{1},\zeta,\zeta^2,\ldots \}}$.}
		
		\revision{
			``$\Leftarrow$''. Assume that the Szeg{\H o} condition holds. Then, by Theorem~\ref{thm:KolmogorovDensity} we have that $L^2_\nu = \overline{\mathrm{span}\{\mathds{1},\zeta,\zeta^2,\ldots \}}$, and in particular $\zeta^{-1} \in \overline{\mathrm{span}\{\mathds{1},\zeta,\zeta^2,\ldots \}}$.}
		
		\revision{
			``$\Rightarrow$''. The proof of this direction is due to Friedrich Philipp and is inspired by \cite[Lemma~7.4]{philipp2017bessel}. We include it with his permission.}
		
		\revision{
			Assume weak predictiveness, i.e., $\zeta^{-1} \in \overline{\mathrm{span}\{\mathds{1},\zeta,\zeta^2,\ldots \}}$. By \cite[Lemma~2.11.3]{simon2010szegHo} we have that $\zeta^{-n} \in \overline{\mathrm{span}\{\mathds{1},\zeta,\zeta^2,\ldots \}}$ for every $n\in\mathbb{N}$. In particular, for any two polynomials $p,q$ we have that $p(\zeta) + q(\bar\zeta) \in \overline{\mathrm{span}\{\mathds{1},\zeta,\zeta^2,\ldots \}}$, showing that $\overline{\mathrm{span}\{\mathds{1},\zeta,\zeta^2,\ldots \}}$ is closed under complex conjugation (here, $\bar\zeta$ stands for the complex conjugate of~$\zeta$). The Stone--Weierstra{\ss} theorem now implies that $\overline{\mathrm{span}\{\mathds{1},\zeta,\zeta^2,\ldots \}}$ is dense in~$C(\bT)$, the usual space of continuous functions on the unit circle. Since $\nu$ is a Borel probability measure on $\bT$, it is regular, and \cite[Theorem~3.14]{Rud87} shows that $C(\bT)$ is dense in~$L^2_\nu$. Theorem~\ref{thm:KolmogorovDensity} now implies that the Szeg{\H o} condition holds.}
	\end{proof}

    \revision{The above proof also shows that weak predictiveness implies that $\cU^n f$ can be approximated from past observations for \emph{every}~$n\in\mathbb{N}$, not merely for~$n=1$.}
	
	\begin{remark}[Time flip]
		\revision{
			Note that swapping the roles of $T$ and $T^{-1}$ is merely a reflection of the trace measure $\nu$ on the real axis. Thus, all statements on weak predictiveness that we consider here are valid both for $T$ and~$T^{-1}$. }
		
		\revision{
			Furthermore, the proof of Theorem~\ref{thm:weakpred} shows that weak predictiveness is equivalent to the cyclic subspace of $f$ under $\cU$ being equal to its cyclic subspace under $\cU^{-1}$, i.e., $C_{f,\cU} = C_{f,\cU^{-1}}$. To see this, first note that the proof shows the equivalence of weak predictiveness and $\overline{\mathrm{span}\{\ldots, \zeta^{-2},\zeta^{-1},\mathds{1},\zeta,\zeta^2,\ldots \}} = \overline{\mathrm{span}\{\mathds{1},\zeta,\zeta^2,\ldots \}} = L^2_\nu$. By complex conjugation one obtains that also $\overline{\mathrm{span}\{\mathds{1},\zeta^{-1},\zeta^{-2},\ldots \}} = L^2_\nu$, which implies our claim.
		}
	\end{remark}
	
	If $\cU$ has discrete spectrum, then for any observable $f$ the decomposition of the associated trace measure $\nu$ in Theorem~\ref{thm:KolmogorovDensity} has no absolutely continuous part, i.e.,~$w\equiv 0$. Theorem~\ref{thm:weakpred} thus implies:
	\begin{proposition}
		\label{prop:DiscSpecWeakPredict}
		If the invertible mpt~$T$ has discrete spectrum, then every observable $f\in L^2_\mu$ is weakly predictive.
	\end{proposition}
	
	We note that Theorem~\ref{thm:weakpred} allows for mixing systems (that have purely continuous spectrum) also to be predictive; for instance by having a subinterval of $\bT$ (i.e., an arc) of positive (arc-length) measure on which~$w=0$. Note that this is compatible with our observation in section~\ref{subsec:RankOne} below, that there are mixing systems which even possess cyclic vectors.

	\section{Genericity of asymptotically linearly predictive pairs~$(T,f)$}
	\label{sec:genericity}
	
	Predictiveness in Proposition~\ref{prop:forecasterror} requires cyclic vectors of~$\cU^{-1}$. For simplicity of notation, in the remainder of this section we are going to speak about cyclic vectors of~$\cU$. We ask the reader to keep in mind that the statements hold true for~$\cU^{-1}$ just as well, because the properties discussed herein are shared by~$T$ and~$T^{-1}$. 
	
	In section~\ref{subsec:basicsMT} below we define $\mpt$ as the set of all inveritble mpts and endow it with the weak topology.
	We are going to investigate the size of the set of pairs $(T,f)$ such that $f$ is a cyclic vector of~$\cU = \cU_T$.
	For instance, we show in section~\ref{ssec:conjugacy_cyclic} that a dense set of measure-preserving transformations admits a prevalent set of cyclic vectors. The collection of mpts with a dense {\Gd} set of cyclic vectors is even a dense {\Gd} set in $\mpt$ (see section~\ref{subsec:BaireMethod}). This result is extended to the setting of measure-preserving homeomorphisms in section~\ref{subsec:homeo}. 
	
	In section~\ref{subsec:RankOne} we discuss so-called rank one systems. We will see that the rank one property is generic in $\mpt$ and guarantees existence of cyclic vectors. 
	
	\subsection{\label{sec:Preliminaries}Preliminaries}
	
	\subsubsection{Some terminology from topology}
	\label{subsec:basicsTop}

	\begin{definition}[{\Gd} and generic]
		\label{def:Gd}
		In a topological space a \emph{{\Gd} set} is a countable intersection of open sets.
		We call a property \emph{generic} if the set with this property contains a dense {\Gd} set.
	\end{definition}
	Generic properties represent sets which are large in a topological sense, and in particular, nonempty.
	Since a unitary transformation is an isometric bijection, it maps \Gd{} sets to \Gd{} sets. It is also easy to see that a unitary transformation maps dense sets to dense ones; we will spell this out in Lemma~\ref{lem:cyclic_match} for a specific setting. 
	
	For comparison, we also mention another related concept of genericity: Some authors define a generic property as one that holds on a \emph{residual set} (that is, a countable intersection of dense open sets), with the dual concept being a \emph{meagre set} (i.e., a countable union of nowhere dense closed sets). We observe as an important practical aspect for applications that, if a property holds on a residual set, it may not hold for every point, but perturbing the point slightly will generally land one inside the residual set (by nowhere density of the components of the meagre set). Thus, generic sets constitute the most important case to address in theorems and algorithms. 
	
	In this paper we will actually study generic sets in completely metrizable topological spaces. In those spaces a residual set is dense by Baire category theorem and, hence, the two concepts of genericity coincide.
	
	The next ``largeness'' concept, prevalence, arose from the desire to generalize the notion ``Lebesgue-almost surely'' to infinite-dimensional spaces.
	
	\begin{definition}[Prevalence \cite{SaYoCa91}]
		\label{def:prevalence}
		Let $\set{V}$ be a vector space. A set $\set{S} \subset \set{V}$ is called \emph{prevalent} if there is a finite-dimensional \emph{probe space} $\set{E} \subset \set{V}$ such that for all $v\in\set{V}$ (Lebesgue-)almost every point in $v + \set{E}$ belongs to~$\set{S}$.
	\end{definition}
	
	We also note that unitary transformations map prevalent sets to prevalent sets. \revision{To see this, we pick an orthonormal basis $\{e_1,\dots, e_{\dim\set{E}}\}$ of the finite-dimensional probe space $\set{E}$ and define the Lebesgue measure $\mu_{v+\set{E}}$ on $v+\set{E}$ via the coefficients with respect to this basis, that is, for any subset $A$ of $v+\set{E}$ we define $\mu_{v+\set{E}}(A)=\lambda(\vec{A})$, where $\smash{ \vec{A}=\Meng{\vec{c}\in \mathbb{C}^{\dim\set{E}}}{v+\sum^{\dim \set{E}}_{i=1}c_ie_i \in A} }$ and $\lambda$ is the Lebesgue measure on $\mathbb{C}^{\dim \set{E}}$. Since a unitary transformation $U$ is isometric, it is measure-preserving as a map from $(v+\set{E},\mu_{v+\set{E}})$ to $(Uv+U\set{E},\mu_{Uv+U\set{E}})$.} 
	
	\subsubsection{\label{subsec:basicsMT}Some terminology from measure theory}
	Two measure spaces $(\mathbb{X},\mathcal{B},\mu)$ and $(\mathbb{X}',\mathcal{B}',\mu')$ are said to be \emph{isomorphic} if there are null sets $N \subset \mathbb{X}$, $N' \subset \mathbb{X}'$ and an invertible map $\phi:\mathbb{X}\setminus N \to \mathbb{X}'\setminus N'$  such that $\phi$ and $\phi^{-1}$ both are measurable and measure-preserving maps. Such a map $\phi$ is called a \emph{measure-theoretic isomorphism}. In the case when $(\mathbb{X},\mathcal{B},\mu)=(\mathbb{X}',\mathcal{B}',\mu')$ one sometimes also calls such an isomorphism $\phi$ an \emph{automorphism}. We denote the collection of automorphisms of a given measure space by $\mpt$. As usual, two measure-preserving transformations are identified if they differ on a set of measure zero only.
	
	If $\mu(\mathbb{X})=1$, then the measure space $(\mathbb{X},\mathcal{B},\mu)$ is a probability space. An important class of probability spaces are the \emph{Lebesgue--Rokhlin spaces}, that is, they are complete and separable probability spaces. We refer to \cite[section 9.4]{Boga} for details. Every Lebesgue--Rokhlin spaces is isomorphic to a disjoint union of the unit interval with Lebesgue measure and at most countably many atoms \cite[Theorem 9.4.7]{Boga}. 
	
	In the following, we restrict ourselves to Lebesgue--Rokhlin spaces without any atoms and call them \emph{standard measure spaces}. Since every standard measure space is isomorphic to the unit interval with Lebesgue measure on the Borel sets, every invertible mpt of a standard measure space is isomorphic\footnote{Two mpts $T:\mathbb{X} \to \mathbb{X}$ and $S: \mathbb{X}' \to \mathbb{X}'$ are called \emph{isomorphic} if there is an isomorphism $\phi:\mathbb{X}\to \mathbb{X}'$ such that $S = \phi \circ T \circ \phi^{-1}$.} to an invertible Lebesgue-measure-preserving transformation on~$[0,1]$. Hence, it often suffices to prove statements for mpts on~$[0,1]$.
	
	We endow $\mpt$ with the \emph{weak topology}. 
	\begin{definition}[Weak topology on $\mpt$ {\cite[p.~62]{Hal17}}]
		A subbasis of the weak topology is given by all sets of the form
		\[
		N(T,E,\ep) := \{S \in \mathrm{MPT} \mid \mu(TE \, \triangle \, SE) < \ep \},
		\]
		where $E$ is measurable and~$\ep>0$. Convergence in the weak topology of the sequence of mpts $(T_n)_n$ to the mpt $T$, denoted by $T_n \stackrel{w}{\to} T$, is thus defined such that $\mu(T_n E \, \triangle \, T E) \to 0$ as $n\to \infty$ for every $E\in \mathfrak{B}$.
	\end{definition}
	
	This weak topology is even completely metrizable, see \cite[p.~64]{Hal17}. To quote Halmos: ``Since, however, this rather artificial construction [of the metric] does not seem to throw any light on the structure of $\mpt$, I see no point in studying it further.''
	Weak convergence of mpts can be characterized as follows. The proof can be found in Appendix~\ref{app:aux_proofs}.
	\begin{lemma}
		\label{lem:weakapprox}
		For $T,T_n \in \mpt$, $n\in\N$, the following are equivalent:
		\begin{enumerate}
			\item $T_n \stackrel{w}{\to} T$ as $n\to\infty$.
			\item $T_n^{-1} \stackrel{w}{\to} T^{-1}$ as $n\to\infty$.
			\item $\cU_{T_n} \to \cU_T$ as $n\to\infty$ in the strong operator topology, i.e., $\| \cU_{T_n} g - \cU_T g\|_{L^2_\mu} \to 0$ for all~$g\in L^2_\mu$.
			\item $\cU_{T_n} \to \cU_T$ as $n\to\infty$ in the weak operator topology, i.e., $\langle \cU_{T_n} g - \cU_T g, h \rangle_{L^2_\mu} \to 0$ for all~$g,h\in L^2_\mu$.
			\item For every $g\in L^2_\mu$, one has that $\| g\circ T_n^i - g\circ T^i\|_{L^2_\mu} \to 0$ as $n\to\infty$ for every~$i\in\mathbb{Z}$.
		\end{enumerate}
	\end{lemma}

	\subsection{Ergodic circle rotations have a prevalent set of cyclic vectors}
	
	Recall that predictiveness requires cyclic vectors of~$\cU^{-1}$, while, for simplicity of notation, we consider cyclic vectors of $\cU$ in the entire section~\ref{sec:genericity}. Note that the statement of the next result holds just as well for~$T^{-1}$ as for~$T$.
	\revision{Let $S^1=\R/\mathbb{Z}$ denote the unit circle (the onedimensional unit torus) closed under addition.}
	
	\begin{proposition}\label{prop:Rotation}
		Let $T:S^1\to S^1$, $Tx = x + \alpha \mod 1$, be an ergodic circle rotation and~$\cU:L^2(S^1)\to L^2(S^1)$ the associated Koopman operator, where the measure underlying $L^2(S^1)$ is the Lebesgue measure. Then $\cU$ has a prevalent set of cyclic vectors. Furthermore, the set of cyclic vectors contains a dense \Gd{} set.
	\end{proposition}
	\begin{proof}
		
		\textbf{1. Prevalence.}
		We recall that the translation semigroup $\mathcal{T}^t f = f(\cdot - t)$, $t\in\R$, is strongly continuous on~$L^2(S^1)$, which means that $\lim_{t\to s} \mathcal{T}^t f = \mathcal{T}^s f$ for~$f\in L^2(S^1)$, cf.~\cite[I.4.18]{EN00}. This implies, since the sequence $(n\alpha)_n$ is dense in $S^1$ by ergodicity of~$T$, that
		\begin{equation}
			\label{eq:f_translations}
			\liminf_{n\to\infty} \| \cU^n f - f(\cdot - t)\|_{L^2} = 0\quad \forall t\in S^1.
		\end{equation}
		Assume that $f\in L^2(S^1)$ is not cyclic. Then there is a $0 \neq g\in L^2(S^1)$ such that
		\[
		g \perp \overline{\mathrm{span}\{f,\cU f,\cU^2 f,\ldots \}},
		\]
		which is by~\eqref{eq:f_translations} equivalent to
		\begin{equation}
			\label{eq:g_perp}
			0 = \int_{S^1} g\, f(\cdot - t) = (g * \refl f)(t) \quad \forall t\in S^1,
		\end{equation}
		where $\refl f$, defined by $(\refl f)(x) := f(-x)$ for $x\in S^1$, is the reflection of $f$ and $g*h$ denotes the convolution of $g$ with~$h$. Denoting by $\FT: L^2(S^1) \to \ell^2(\mathbb{Z})$ the Fourier transform, we have that \eqref{eq:g_perp} is equivalent to
		\begin{equation}
			\label{eq:g_perp_fourier}
			\FT g(k) \cdot \FT \refl f(k) = 0 \quad \forall k\in \mathbb{Z}.
		\end{equation}
		Consider the probe vector $p \in L^2(S^1)$ such that the Fourier transform of its reflection is given by
		\[
		\FT \refl p(k) := \left\{ \begin{array}{ll}
			1, & k=0,\\
			\frac{1}{k^2}, & k\neq 0.
		\end{array}\right.
		\]
		Alternatively, one could consider the real-valued probe vector $p$ by replacing $\frac{1}{k^2}$ by~$\frac{1}{k^2}(1 + i\, \mathrm{sign}(k))$, where $i$ denotes the imaginary unit. Let $\lambda\in\C$. By \eqref{eq:g_perp_fourier}, the requirement for $f+\lambda p$ \emph{not} to be cyclic is that
		\[
		\FT g(k) \cdot \left( \FT \refl f(k) + \lambda \FT \refl p(k) \right) = 0\quad \forall k\in \mathbb{Z}
		\]
		for some (possibly $\lambda$-dependent)~$g\neq 0$. However, $g\neq 0$ implies that $\FT g(k) \neq 0$ for at least one~$k\in\mathbb{Z}$. Thus, there is a $k\in\mathbb{Z}$ with $\FT \refl f(k) + \lambda \FT \refl p(k) = 0$, implying~$\lambda = -k^2 \FT \refl f(k)$. Hence, cyclicity of $f+\lambda p$ can fail at most at countably many~$\lambda$. In other words, $f+\lambda p$ is cyclic for Lebesgue-almost every~$\lambda$, rendering the set of cyclic vectors of $\cU$ prevalent.
		
		\textbf{2. Dense \Gd.}
		By the above we have the implication
		\begin{equation}
			\label{eq:cyclic_impl}
			f \text{ not cyclic} \quad\Longrightarrow\quad \exists \, k\in\mathbb{Z} \text{ such that } \FT \refl f (k)=0.
		\end{equation}
		For $k\in\mathbb{Z}$ let us define 
		\[
		A_k := \left\{ f \in L^2(S^1) \mid \FT \refl f (k) \neq 0 \right\} = (\FT \refl)^{-1} \left( \left\{  s\in \ell^2(\mathbb{Z}) \mid s_k \neq 0\right\} \right),
		\]
		where we denote by $s_k$ the $k$-th element of a sequence~$s$. With \eqref{eq:cyclic_impl} we have that $\bigcap_{k\in\mathbb{Z}} A_k$ is contained in the set of cyclic vectors.
		Note that both $\FT$ and $\refl$ are unitary transformations, hence $(\FT\refl)^{-1}$ is a bounded operator preserving openness and denseness with respect to the norm topologies. As the sets $\left\{  s\in \ell^2(\mathbb{Z}) \mid s_k \neq 0\right\}$ are open and dense for every $k\in\mathbb{Z}$, the set $\bigcap_{k\in\mathbb{Z}} A_k$ is~\Gd. Since the set $\left\{  s\in \ell^2(\mathbb{Z}) \mid s_k \neq 0\ \forall k \in\mathbb{Z} \right\}$ is also dense, so is its preimage $\bigcap_{k\in\mathbb{Z}} A_k$ under~$\FT\refl$. 
	\end{proof}
	
	\begin{remark}
		Note that the set $\left\{  s\in \ell^2(\mathbb{Z}) \mid s_k \neq 0\ \forall k \in\mathbb{Z} \right\}$ is not open in~$\ell^2(\mathbb{Z})$.
		Thus, ``dense \Gd'' cannot be replaced in the statement of Proposition~\ref{prop:Rotation} by ``open and dense''.
	\end{remark}

	\subsection{One map with a prevalent set of cyclic vectors implies densely many}
	\label{ssec:conjugacy_cyclic}
	
	We just showed that ergodic circle rotations possess a prevalent set of cyclic vectors. Through the isomorphy of the nonatomic and separable probability space $(\set{X},\mu)$ and $([0,1], \mathrm{Leb})$ (see section~\ref{subsec:basicsMT} and also \cite[p.~61]{Hal17}) we can build an ergodic ``rotation'' on $(\set{X},\mu)$, hence a transformation with a prevalent set of cyclic vectors. 
	First, one transfers a cyclic vector via isomorphy:
	\begin{lemma}
		\label{lem:cyclic_match}
		Let $T:\set{X}\to \set{X}$ be a $\mu$-preserving transformation with a cyclic vector~$c$. For every $\mu$-preserving essential bijection $S:\set{X}\to \set{X}$, the mpt $\widetilde{T} = S^{-1}TS$ has the cyclic vector~$c\circ S$. 
	\end{lemma}
	\begin{proof}
		Note that $\cU_{\widetilde{T}}^n f = \cU_S \circ \cU_T^n \circ \cU_{S}^{-1}$for~$n\in\N$. In particular,
		\[
		\cU_{\widetilde{T}}^n (c \circ S) = \cU_S (\cU_T^n c).
		\]
		Since $\mathrm{span}\{\cU_T^n c\mid n\in\N\}$ is dense, it is sufficient to show that for any dense subspace $\smash{ \set{U} \subset L^2_\mu }$ also the space $\smash{ \cU_{S}\set{U} }$ is dense. For any fixed $\ep>0$ and $\smash{ f\in L^2_\mu }$ let $u\in\set{U}$ be such that~$\smash{ \| u- f\circ S^{-1} \|_{L^2_\mu} < \ep }$. Since $S$ is measure-preserving and an essential bijection, we have that~$\smash{ \| u\circ S - f \|_{L^2_\mu} < \ep }$. Thus, $\cU_{S}\set{U}$ is dense and the claim follows.
	\end{proof}
	
	To show denseness of transformations with specific ergodic properties, the following tool is extremely useful.
	
	\begin{lemma}[Conjugacy Lemma, {\cite[p.~77]{Hal17}}]
		\label{lem:conjugacy}
		Let $T_0$ be an aperiodic $\mu$-preserving transformation (i.e., the set of periodic points has zero measure). In the weak topology the set of all transformations of the form $S^{-1}T_0S$, where $S$ is a $\mu$-preserving essential bijection, is dense in~$\mpt$.
	\end{lemma}
	
	It follows that the set of transformations that has prevalently many cyclic vectors is dense in MPT, because we can match every cyclic vector $c$ of $T$ with a cyclic vector $c\circ S$ of~$S^{-1}TS$. The mapping $c \mapsto c\circ S$ is a unitary transformation, since $S$ is a measure-preserving essential bijection, and thus it preserves the properties discussed in section~\ref{subsec:basicsTop}.
	
	Altogether we conclude the following denseness result.
	\begin{theorem}\label{thm:DensePrevalent}
		The collection of $\mu$-preserving transformations with a prevalent and dense \Gd{} set of cyclic vectors in~$L^2_{\mu}$ lies dense in~$\mpt$ with respect to the weak topology.
	\end{theorem}
	
	\begin{proof}
		We take any fixed irrational rotation $R_{\alpha}$ on $S^1$ (that has a dense \Gd{} and prevalent set of cyclic vectors by Proposition~\ref{prop:Rotation}) and conjugate it with an isomorphy $\phi$ from $\mathbb{X}$ to $S^1$ to get an aperiodic mpt $T_0=\phi^{-1} \circ R_{\alpha} \circ \phi$ on~$\mathbb{X}$. By Lemma~\ref{lem:cyclic_match}, $T_0$ has a dense \Gd{} and prevalent set of cyclic vectors. Then we apply the Conjugacy Lemma (Lemma~\ref{lem:conjugacy}) to get a dense set of mpts, each of which have a dense \Gd{} and prevalent set of cyclic vectors, again by Lemma~\ref{lem:cyclic_match}.
	\end{proof}
	
	\begin{remark}[Weakly predictive mpts are dense]
		In Theorem~\ref{thm:DensePrevalent} we used an ergodic circle rotation to generate densely many ergodic transformations on an arbitrary Lebesgue--Rokhlin space $(\mathbb{X},\mathcal{B},\mu)$ without atoms, such that all these transformations yield a prevalent set of cyclic vectors.
		Note that for such spaces $L^2_\mu(\set{X})$ is separable. From \cite[\S3.3]{Wal00} we know that ergodic group rotations have discrete spectrum.
		Thus, also the dense set of transformations constructed in Theorem~\ref{thm:DensePrevalent} have discrete spectrum, and by Proposition~\ref{prop:DiscSpecWeakPredict} all observables for these transformations are weakly predictive.
	\end{remark}
	
	It is natural to ask whether for an ergodic transformation the possession of discrete spectrum already implies that the set of cyclic vectors is prevalent. The Representation Theorem \cite[Theorem~3.6]{Wal00} states that an ergodic measure-preserving transformation with discrete spectrum on a probability space is isomorphic to an ergodic rotation on a compact abelian group. The strong analogy between (ergodic) circle rotations and group rotations, together with the prevalence statement of Proposition~\ref{prop:Rotation} could lead one to formulate the following conjecture:
	\begin{conj}
		\label{conj:discr_spec_prevalence}
		Let the ergodic $T \in \mpt$ possess discrete spectrum. Then $\cU = \cU_T$ has a prevalent set of cyclic vectors in~$L^2_\mu$.
	\end{conj}
	
	In Section~\ref{subsec:RankOne} we will discuss another type of systems that possesses cyclic vectors. 
	
	\subsection{A generic MPT has a dense {\Gd} set of cyclic vectors}\label{subsec:BaireMethod}

	By Theorem~\ref{thm:DensePrevalent} the collection of measure-preserving transformations with a dense \Gd{} set of cyclic vectors lies dense in~$\mpt$ with respect to the weak topology. In fact, we can even prove that it is a dense {\Gd} set in~$\mpt$. We use similar methods as in \cite[Theorem~8.26]{Nad20}, where it is shown that the collection of measure-preserving transformations admitting a cyclic vector is a dense {\Gd} set in $\mpt$ with respect to the weak topology. 
	
	\begin{proposition} \label{prop:DenseGd}
		The collection of $\mu$-preserving transformations with a dense {\Gd} set of cyclic vectors in $L^2_\mu$ is a dense {\Gd} set in $\mpt$ with respect to the weak topology.
	\end{proposition}
	
	\begin{proof}
		By Theorem~\ref{thm:DensePrevalent} we already know that the collection of measure-preserving transformations with a dense {\Gd} set of cyclic vectors is dense in~$\mpt$.
		
		To show that it is a {\Gd} set, we let $\{f_j\}_{j\in \N}$ be a dense set in $L^2_{\mu}$. \revision{For $g\in L^2_{\mu}$, $m,N\in \N$ we introduce the set 
			\begin{align*}
				\tilde{V}_i(m,N,g) & = \Meng{\cU \in \mathcal{B}(L^2_\mu)}{\exists \lambda^{(i)}_k \in \C,\, k=0,\dots, N: \, \norm{f_i - \sum^N_{k=0}\lambda^{(i)}_k \cU^kg}_{L^2_{\mu}} < \frac{1}{m}} \\
				& = \bigcup_{p \in P_N}  \Meng{\cU \in \mathcal{B}(L^2_\mu)}{\norm{f_i - p(\cU)g}_{L^2_{\mu}} < \frac{1}{m}}, 
			\end{align*}
			where $P_N$ denotes the set of polynomials of degree not larger than $N$. Such a set $\tilde{V}_i(m,N,g)$ is open in the space $\mathcal{B}(L^2_\mu)$ of bounded linear operators with respect to the strong operator topology because $\cU_n \to \cU$ implies $p(\cU_n) \to p(\cU)$ for every polynomial $p$ by Lemma~\ref{lem:weakapprox}(5). Since the weak topology on $\mpt$ coincides with the topology induced from the strong operator topology on Koopman operators (cf.\ Lemma~\ref{lem:weakapprox}), we see that the sets
			\begin{equation*}
				V_i(m,N,g) = \Meng{T \in \mpt}{\exists \lambda^{(i)}_k \in \C,\, k=0,\dots, N: \, \norm{f_i - \sum^N_{k=0}\lambda^{(i)}_k \cU^k_Tg}_{L^2_{\mu}} < \frac{1}{m}}.
			\end{equation*}
			are open with respect to the weak topology.} Thus, the finite intersections 
		\begin{align*}
			& V(m,n,N,g) \coloneqq \bigcap_{i=1}^n V_i(m,N,g) \\
			=&\Meng{T\in \mpt}{\forall i=1,\ldots , n \, \exists \lambda^{(i)}_k \in \C,\,k=1,\ldots,N:\, \norm{f_i - \sum^N_{k=0}\lambda^{(i)}_k \cU^k_Tg}_{L^2_{\mu}} < \frac{1}{m}}
		\end{align*}
		are also open. \revision{With this terminology, the collection of mpts with a dense {\Gd} set of cyclic vectors is given by the {\Gd} set 
			\begin{equation*}
				\mathcal{T} := \bigcap_{m\in \N}\bigcap_{n \in \N}\bigcap_{t\in \N} \bigcap_{s \in \N} \bigcup_{N \in \N} \bigcup_{g \in L^2_{\mu}, \norm{g-f_s}_{L^2_{\mu}}<\frac{1}{t}} V(m,n,N,g),
			\end{equation*}
			where it remains to show that $T\in \mathcal{T}$ has a dense \Gd{} set of cyclic vectors in $L^2_\mu$. Heuristically, this dense {\Gd} set of cyclic vectors is obtained by the intersections over $t$ and $s$ which guarantee for any function $f_s$ from our dense family and any $t\in \N$ the existence of a cyclic vector that is $\frac{1}{t}$-close to~$f_s$.} To show this in detail, let $B_r(f) \subset L^2_\mu$ denote the open ball with radius $r>0$ and center $f\in L^2_\mu$ and consider 
		\begin{align*}
			\mathcal{T}(m,n) &= \bigcap_{t\in \N} \bigcap_{s \in \N} \bigcup_{N \in \N} \bigcup_{g \in L^2_{\mu}, \norm{g-f_s}_{L^2_{\mu}}<\frac{1}{t}} V(m,n,N,g) \\
			&= \left\{ T\in\mpt \,\Big\vert\,
			\begin{array}{l}
				\forall s,t\in\N \ \exists g\in B_{1/t}(f_s) \text{ and } \exists N\in\N \text{ s.t. }\forall i=1,\ldots,n \\
				\exists p\in P_N \text{ with } \norm{p(\cU_T)g - f_i}_{L^2_\mu} < \frac1m
			\end{array}
			\right\}.
		\end{align*}
		The functions $g \in L^2_\mu$ that appear in the characterization of $\mathcal{T}(m,n)$ are candidates for cyclic vectors upon intersection over all~$m,n\in\N$. We denote this set by $C(T,m,n) \subset L^2_\mu$, more precisely, for a fixed $T\in\mathcal{T}(m,n)$ let
		\[
		C(T,m,n) := \bigcap_{t\in\N} \bigcup_{s\in\N} \left[ B_{1/t}(f_s) \cap \left\{ g \in L^2_\mu \,\Big\vert\, 
		\begin{array}{l}
			\exists N\in\N \text{ s.t. } \forall i=1,\ldots,n \ \exists p\in P_N\\
			\text{ s.t. } \norm{p(\cU_T)g - f_i}_{L^2_\mu} < \frac1m
		\end{array}
		\right\} \right].
		\]
		For $T \in \mathcal{T} = \bigcap_{m,n \in\N} \mathcal{T}(m,n)$, the set of cyclic vectors is given by~$\bigcap_{m,n\in\N} C(T,m,n)$. To see that this is a \Gd{} set, it is sufficient to show that the $C(T,m,n)$ are themselves \Gd{} for every~$m,n\in\N$. This is straightforward by noting that $p(\cU_T)$ is a continuous mapping on~$L^2_\mu$ for every~$p\in P_N$ and that
		\[
		C(T,m,n) = \bigcap_{t\in\N} \bigcup_{s\in\N} \left[ B_{1/t}(f_s) \cap \bigcap_{i=1}^n \bigcup_{N\in\N} \bigcup_{p\in P_N} p(\cU_T)^{-1} B_{1/m}(f_i) \right].
		\]
		Arbitrary unions and finite intersections of open sets are open, hence $C(T,m,n)$ and with it the set of cyclic vectors for $T \in \mathcal{T}$ is \Gd{}, concluding the proof.
	\end{proof}
	
	While this result gives us the genericity of transformations with cyclic vectors, it does not provide us with a criterion to check if a given transformation has cyclic vectors. In the following section we meet such a sufficient condition to guarantee the existence of cyclic vectors.
	
	\subsection{Rank one systems}\label{subsec:RankOne}
	
	In this section we discuss \emph{rank one systems} which constitute a generic class of measure-preserving transformations possessing cyclic vectors. We start with the admittedly technical definition of rank one systems.\footnote{According to Ferenczi it is the lecturer's nightmare to define rank one systems \cite[p.~40]{ferenczi1997systems}. We refer to his survey article for several equivalent definitions of the rank one property.} 
	
	\begin{definition}[Rank one system]
		\label{def:rankone}
		A system $T:(\mathbb{X},\mu) \to (\mathbb{X},\mu)$ is of \emph{rank one} if for any measurable set $A\subset \mathbb{X}$ and any $\varepsilon>0$ there exist $F\subset \mathbb{X}$, $h\in \mathbb{Z}^+$ and a measurable subset $A^{\prime}\subset \mathbb{X}$ such that
		\begin{itemize}
			\item the sets $T^kF$, $k=0,1,\ldots,h-1$ are disjoint;
			\item $\mu(A\triangle A^{\prime})<\varepsilon$;
			\item $\mu(\bigcup^{h-1}_{k=0}T^kF)>1-\varepsilon$;
			\item $A^{\prime}$ is measurable with respect to the partition $\xi(\mathcal{T})$ formed by the sets $F,TF,\ldots , T^{h-1}F$ and $\mathbb{X} \setminus \bigcup^{h-1}_{k=0}T^kF$.
		\end{itemize}
		One calls $\mathcal{T}=\{F,TF,\ldots , T^{h-1}F\}$ a \emph{Rokhlin tower} (or \emph{column}) of height $h$ and with base $F$. 
	\end{definition}
	Thus, a system is rank one if for any fixed set and a tolerance one can choose some set that has disjoint iterates which achieve filling the space and approximating the first set both within the tolerance.
	Furthermore, one says that a sequence of towers $\mathcal{T}_n$ is \emph{exhaustive} if $\xi(\mathcal{T}_n)$ converges to the decomposition into points as $n\to \infty$, that is, for every measurable
	set $A\subseteq \mathbb{X}$ and for every $n \in \N$ there exists a set $A_n$, which is a union of elements of $\xi(\mathcal{T}_n)$, such that $\lim_{n\to \infty}\mu(A \triangle A_n)=0$. By definition, rank one systems have exhaustive sequences of towers.
	
	If $T\in\mpt$ is rank one, so is~$T^{-1}$. To see this, note that for every $\ep>0$ and associated tower $\mathcal{T}$ with base $F$ for $T$ we get the same tower with base $T^{h-1}F$, now for $T^{-1}$, satisfying the requirements of Definition~\ref{def:rankone}.
	
	Examples of rank one systems include rotations and odometer transformations; see our example in section~\ref{subsec:odometer}. In general, it is true that discrete spectrum together with ergodicity implies rank one \cite{Junco_1976}.
	Rank one systems were introduced and actively studied in the framework of so-called \emph{cutting-and-stacking constructions} in ergodic theory (see e.g.~\cite[section~5.2]{KatokThouvenot}). Here, one can view a rank one system as a transformation obtained from a cutting-and-stacking construction with a single tower but there is a scarcity of explicit examples. 
	
	Our interest in the rank one property is based on the fact that it is a sufficient condition to guarantee the existence of cyclic vectors.\footnote{\label{foot:specmultiplicity}We also mention that there is a general notion of \emph{rank} of a measure-preserving transformation. This rank is always greater than or equal to the spectral multiplicity of the associated Koopman operator. Since we are mainly interested in systems with simple spectrum, we focus on the case of rank one. We refer the interested reader to \cite[chapter~5]{Nad20}.}
	
	\begin{proposition} \label{prop:rank1CyclicNew}
		A rank one system has cyclic vectors.
	\end{proposition}
	
	\begin{proof}
		We combine arguments from \cite[chapter 4]{Nad20} and \cite[Proposition 5.8]{KatokThouvenot}. 
		
		Let $T$ be a rank one transformation. Recall that $\smash{ C_g = \overline{\mathrm{span} \{ g,\cU_T g, \cU^2_T g,\ldots\}} }$ denotes the cyclic subspace of~$g\in L^2_{\mu}$. Furthermore, $d_{L^2_\mu}(f,C_g)$ denotes the $L^2_{\mu}$ distance of $f\in L^2$ from $C_g$. 
		
		Suppose that $T$ does not admit cyclic vectors. Then there exist orthogonal unit vectors $f_1,f_2 \in L^2_{\mu}$ such that 
		\begin{equation}\label{eq:ContraProjDist}
			d^2_{L^2_\mu}(f_1,C_g) + d^2_{L^2_\mu}(f_2,C_g) \geq 1
		\end{equation}
		for every $g\in L^2_{\mu}$ (see \cite[Lemma~3.1]{KatokStepin} or \cite[Theorem~4.3]{Nad20} for a proof of a more general result).  
		As in the definition of the rank one property, we consider an exhaustive sequence of towers $\mathcal{T}_n$ with base $F_n$ and height $h_n$ approximating~$T$. Then for each $n$ the images of the characteristic function $\chi_{F_n}$ under $\cU^i_T$, $i=0,1,\ldots,h_n-1$, are characteristic functions of the disjoint levels of the tower~$\mathcal{T}_n$. Then for each $n$ there is a cyclic subspace $C_{\chi_{F_n}}$ which contains all characteristic functions of the levels of the tower and their linear combinations. Since our sequence of towers is exhaustive, we can choose $n$ sufficiently large such that $d_{L^2_\mu}(f_1,C_{\chi_{F_n}})<\frac{1}{2}$ and $d_{L^2_\mu}(f_2,C_{\chi_{F_n}})<\frac{1}{2}$. We obtained a contradiction to~\eqref{eq:ContraProjDist}. Hence, $T$ has cyclic vectors.
	\end{proof}
	
	In the converse direction to Footnote~\ref{foot:specmultiplicity}, the simplicity of the spectrum does not imply the rank one property; see \cite[section 2.4.2]{ferenczi1997systems} for various types of counterexamples. 
	Interestingly, there are also mixing systems of rank one; see e.g.~\cite[section 1.4.4]{ferenczi1997systems} for a celebrated construction by Ornstein.
	We would like to emphasize the following:
	\begin{remark}[Mixing can be predictive]
		\label{rem:mixingPred}
		On the one hand, measure-theoretic (strong) mixing implies gaining asymptotic independence on the past. Intuitively, this could lead to the impossibility of asymptotic predictiveness, since the observation~$f\circ T^{-k}$ gives ``increasingly less information'' about $f\circ T$ as $k$ grows. On the other hand, the existence of mixing systems that are rank one---and thus possess cyclic vectors---tells us that this intuition is not necessarily correct (recall also the discussion after Proposition~\ref{prop:DiscSpecWeakPredict}). 
	\end{remark}
	
	Given the significance of cyclic vectors for predictiveness, it is an interesting question if a rank one system has a prevalent set of cyclic vectors. We do not provide an answer to this question, but note that if it were positive, then also Conjecture~\ref{conj:discr_spec_prevalence} would be true, as ergodic systems with discrete spectrum are rank one.
	
	We provide a proof of the following folklore result by adapting proofs of similar results in \cite[Theorem 1.1]{KatokStepin} and \cite[Theorem~3.1]{RatErg}.
	
	\begin{proposition} \label{prop:rank1Gd}
		The collection of rank one systems is a dense {\Gd} set in $\mpt$ with respect to the weak topology.
	\end{proposition}
	
	\begin{proof}
		Let $\{E_j\}_{j\in \N}$ be a dense collection of measurable subsets in $[0,1]$. We define
		$\mathcal{A}_n$ as the set of transformations $T\in \mpt$ that admit a tower approximating $E_1,\ldots , E_n$ with accuracy $1/n$. We claim that $\mathcal{R}\coloneqq \bigcap_{n\in \N}\mathcal{A}_n$ is a dense {\Gd} set and coincides with the collection of rank one systems.
		
		To see that every $T\in \mathcal{R}$ has rank one, we let $A$ be a measurable set and $\varepsilon>0$. Since $\{E_j\}_{j\in \N}$ is dense, there is $E_k$ with $\mu(A \triangle E_k)< \varepsilon/2$. Let $n>\max(k,2/\varepsilon)$. Since $T \in \mathcal{A}_n$, there is a tower approximating $E_k$ with accuracy $1/n<\varepsilon/2$. Hence, the tower approximates $A$ with accuracy $\varepsilon$. 
		
		In the next step we show that each $\mathcal{A}_n$ is dense. For this purpose, we choose $K\in \N$ sufficiently large such that for all $k\geq K$ each $E_1,\ldots , E_n$ can be approximated by unions of dyadic intervals 
		\[
		I_{k,i}\coloneqq \Bigg[\frac{i}{2^k},\frac{i+1}{2^k}\Bigg), \ i=0,1,\ldots , 2^k-1,
		\]
		with accuracy $1/n$. On the partition $\Meng{I_{k,i}}{i=0,1,\ldots , 2^k-1}$ we consider cyclic permutations of order $k$, that is, maps cyclically permuting the dyadic intervals of rank $k$ by translations. We denote the collection of cyclic permutations of order $k$ by $\Pi_k$. In particular, each such permutation admits a tower of exactly the dyadic intervals of rank $k$. Hence, $\Pi_k \subset \mathcal{A}_n$. 
		By the so-called \emph{Weak Approximation Theorem} from \cite[p.~65]{Hal17} the collection $\bigcup_k \Pi_k$ of cyclic permutations is dense in $\mpt$ with respect to the weak topology. Since all but finitely many of the cyclic permutations are in $\mathcal{A}_n$, $\mathcal{A}_n$ still contains a dense set.
		
		One can also verify that each $\mathcal{A}_n$ is open. Then the Baire category theorem yields that $\mathcal{R}$ is a dense {\Gd} set.
	\end{proof}

	\subsection{A generic measure-preserving homeomorphism has a dense {\Gd} set of cyclic vectors} \label{subsec:homeo}
	
	In this section we want to extend our results to the setting of measure-preserving homeomorphisms. Let $\set{M}$ be a compact connected manifold and $\mu$ be a so-called \emph{OU measure}\footnote{The name OU measure was coined in \cite{Alpern} in honor of Oxtoby and Ulam. Sometimes these measures are also called \emph{good measures}.}, that is, $\mu$ is nonatomic, locally positive (i.e., it is positive on every nonempty open set), and zero on the manifold boundary. We consider the collection $\homeo(\set{M},\mu)$ of $\mu$-preserving homeomorphisms of $\set{M}$ with the uniform topology defined by the complete metric 
	\[
	\mathrm{d}_{\text{unif}}(f,g)=\norm{f-g}_{\text{unif}}=\max_{x\in \set{M}}\,\mathrm{d}\left(f(x),g(x)\right)+\max_{x\in \set{M}}\,\mathrm{d}\left(f^{-1}(x),g^{-1}(x)\right),
	\]
	where $\mathrm{d}$ is a given metric on $\set{M}$ compatible with its topology.\footnote{\revision{Note that Urysohn's metrization theorem implies that every manifold is metrizable. On compact spaces, all such metrics inducing the topology are uniformly equivalent. Hence, they also induce the same uniform topology on $\homeo(\set{M},\mu)$.}} 
	
	In this setting, any measure-theoretic property which is generic for abstract measure-preserving transformations is also generic for measure-preserving homeomorphisms of compact manifolds. This can be stated as follows.
	
	\begin{lemma}[{\cite[Corollary 10.4]{Alpern}}] \label{lem:HomeoGd}
		Let $\mu$ be an OU measure on a compact connected manifold $\set{M}$. Let $\mathcal{V}$ be a conjugate invariant dense {\Gd} subset of $\mpt$ with
		the weak topology. Then $\mathcal{V}\cap \homeo(\set{M},\mu)$ is a dense {\Gd} subset of $\homeo(\set{M},\mu)$ with the uniform topology.
	\end{lemma}
	
	Since the collection of mpts with a dense {\Gd} set of cyclic vectors is conjugate invariant by Lemma~\ref{lem:cyclic_match}, we can apply Lemma~\ref{lem:HomeoGd} on Proposition~\ref{prop:DenseGd}. This gives the following genericity result for homeomorphisms with cyclic vectors.
	
	\begin{proposition}
		Let $\set{M}$ be a compact connected manifold and $\mu$ be an OU measure. The collection of $\mu$-preserving homeomorphisms on $\set{M}$ with a dense {\Gd} set of cyclic vectors in $L^2_\mu$ is a dense {\Gd} set in $\homeo(\set{M},\mu)$ with the uniform topology.
	\end{proposition}
	
	Thus, a generic homeomorphism admits a generic set of predictive $L^2$ observables.
	Note that Takens-type theorems require observables to be continuous.
	
	\section{Properties of the least squares filter}
	\label{sec:filter_properties}
	
	\subsection{Forecast properties}
	
	Beyond the prediction error considered in Proposition~\ref{prop:forecasterror}, a further dynamic object of interest is the autocorrelation function; that is
	\[
	A(n) := \langle \cU^n f, f \rangle_{L^2_\mu} = \int f\circ T^n\, f\, d\mu.
	\]
	It turns out that autocorrelations for $n\le d$ are always reproduced correctly, irrespective of whether the observation function is a cyclic vector of $\cU$ or not.
	
	\revision{
		To show this, let us define the shorthand notations $\Pi_{\mathcal{V}}$ (or $\Pi_h$) for the orthogonal projection in $L^2_{\mu}$ to a subspace $\mathcal{V} \subset L^2_{\mu}$ (or to the onedimensional subspace spanned by $h\in L^2_\mu$) and $\mathcal{K}_d^- := \mathrm{span}\{ f\circ T^{-d+1},\ldots, f\circ T^{-1}\}$. Note, in particular, that~$\mathcal{K}_d^- = \cU^{-1} \mathcal{K}_{d-1}$ and~$\mathcal{K}_d = \mathcal{K}_d^- + \mathrm{span}\{f\}$. Let $\Pi_{\mathcal{V}}^{\perp}:= \mathrm{id} - \Pi_{\mathcal{V}}$ denote the orthogonal projection on the orthogonal complement of~$\mathcal{V}$. }
	\revision{
		\begin{proposition}
			\label{prop:autocorr}
			Let $A_d(n) := \langle \cU_d^n f, f \rangle_{L^2_\mu}$ denote the filter's autocorrelation. We have for every map $T$ and observable $f$ that
			\[
			|A(n) - A_d(n)| \le \left\{
			\renewcommand{\arraystretch}{1.5}
			\begin{array}{ll}
				0, & n \le d, \\
				\| \cU_d f - \cU f \|_{L^2_\mu} \| f \|_{L^2_\mu}, & n = d+1,\\
				(n-d)\frac{\big\| \Pi_{\mathcal{K}_d}^{\perp} \cU f \big\|_{L^2_\mu}}{ \big\| \Pi_{\mathcal{K}_d^-}^{\perp} f \big\|_{L^2_\mu}} \, \| f \|_{L^2_\mu}^2, & n \ge d+2.
			\end{array}
			\right.
			\]
		\end{proposition}
	}
	\begin{proof}
		\revision{
			We can express the statement of Proposition~\ref{prop:L2projU} as
			\begin{equation}
				\label{eq:vareq2}
				\langle \cU_d g - \cU g, f\circ T^{-i} \rangle_{L^2_\mu} = 0\qquad \forall g \in \cK_d,\ i=0,\ldots,d-1.
			\end{equation}
			For $n=1$, the claim $A(n) = A_d(n)$ translates to
			\[
			\langle \cU_d f - \cU f, f\rangle_{L^2_\mu} = 0.
			\]
			Note that \eqref{eq:vareq2} implies this with $g=f$ and $i=0$. By expressing the difference $\cU_d^n - \cU^n$ as a telescopic sum, we can extend this result to higher powers~$n$:
			\begin{align}
				A_d(n) - A(n) &= \sum_{i=0}^{n-1} \left\langle \cU^i \cU_d^{n-i} f - 
				\cU^{i+1} \cU_d^{n-i-1}f, f \right\rangle_{L^2_\mu} \nonumber \\
				&= \sum_{i=0}^{n-1} \langle \cU_d^{n-i} f - 
				\cU \cU_d^{n-1-i}f, f\circ T^{-i}\rangle_{L^2_\mu} \nonumber \\
				&= \sum_{i=0}^{n-1} \langle (\cU_d - \cU) 
				\cU_d^{n-1-i}f, f\circ T^{-i}\rangle_{L^2_\mu} \nonumber \\
				&= \sum_{i=0}^{d-1} \langle (\cU_d - \cU) 
				\cU_d^{n-1-i}f, f\circ T^{-i}\rangle_{L^2_\mu} \nonumber \\
				&{} \qquad +\sum_{i=d}^{n-1} \left\langle \cU^i (\cU_d - 
				\cU ) \cU_d^{n-1-i}f, f \right\rangle_{L^2_\mu}, \label{eq:autocorr_telescope}
			\end{align}
			where we used $\smash{ \langle \cU^i h, f\rangle_{L^2_\mu} = \langle h, f\circ T^{-i} \rangle_{L^2_\mu} }$ for any~$h \in L^2_\mu$, implied by the invariance of~$\mu$.
			Since $i\le n-1$, note that $g:= \cU_d^{n-i-1}f \in \cK_d$ and \eqref{eq:vareq2} implies that for $i=0,\ldots,d-1$,
			\[
			\langle (\cU_d - \cU) 
			\cU_d^{n-i-1}f, f\circ T^{-i}\rangle_{L^2_\mu} = 0.
			\]
			The first term on the right-hand side of \eqref{eq:autocorr_telescope} is hence zero.
			Thus, $A_d(n) = A(n)$ for~$n\le d$, as the second sum on the right-hand side of \eqref{eq:autocorr_telescope} is empty.}
		
		\revision{
			For $n=d+1$ the second sum on the right-hand side of \eqref{eq:autocorr_telescope} reduces to
			\[
			\left\langle \cU^d (\cU_d - \cU ) \cU_d^{0}f, f \right\rangle_{L^2_\mu}.
			\]
			The inequality for $n=d+1$ thus follows with~$\|\cU\|=1$.
			Note that since $\cU$ is unitary and $\cU_d$ is the orthogonal projection of $\cU$, by the non-ex\-pan\-sive\-ness of orthogonal projections we have that $\|\cU_d\| \le 1$.}
		
		\revision{
			For $n\ge d+2$ we can estimate the terms in the second sum on the right-hand side of \eqref{eq:autocorr_telescope} by
			\begin{equation}
				\label{eq:autocorr_bound}
				\left| \left\langle \cU^i (\cU_d - 
				\cU ) \cU_d^{n-1-i}f, f \right\rangle_{L^2_\mu} \right| \le \| \cU \|^i \left\| (\cU_d - \cU)\vert_{\cK_d} \right\| \|\cU_d\|^{n-1-i} \|f\|_{L^2_\mu}^2.
			\end{equation}
			It remains to bound~$\left\| (\cU_d - \cU)\vert_{\cK_d} \right\|$. We note that for every $f^- \in \cK_d^-$ we have that $\cU_d f^- = \cU f^-$ because $\cU f^- \in \cK_d$. Thus, we can write
			\begin{equation}
				\label{eq:UdMinusU_bound}
				\left\| (\cU_d - \cU)\vert_{\cK_d} \right\| = \sup_{\substack{g \in \cK_d\\ g\neq 0}} \frac{\| \cU_d g - \cU g \|_{L^2_\mu}}{\|g\|_{L^2_\mu}} = \sup_{\substack{g = f + f^-\\ f^- \in \cK_d^-}} \frac{\| \cU_d g - \cU g \|_{L^2_\mu}}{\|g\|_{L^2_\mu}},
			\end{equation}
			where the second equation holds because for any $g$ without a contribution from $f$ (i.e., $g\in\cK_d^-$) the fraction on the right is zero, clearly dominated by the supremum as written. We obtain
			\[
			\left\| (\cU_d - \cU)\vert_{\cK_d} \right\| = \sup_{f^- \in \cK_d^-} \frac{\| \cU_d f - \cU f \|_{L^2_\mu}}{\|f+f^-\|_{L^2_\mu}} = \sup_{f^- \in \cK_d^-} \frac{\| \Pi^{\perp}_{\cK_d} \cU f \|_{L^2_\mu}}{\|f+f^-\|_{L^2_\mu}} = \frac{\| \Pi^{\perp}_{\cK_d} \cU f \|_{L^2_\mu}}{\| \Pi^{\perp}_{\cK_d^-} f \|_{L^2_\mu}}.
			\]
			The claim follows by applying the bound \eqref{eq:autocorr_bound} to every term in the sum.}
	\end{proof}
	\revision{Thus, if $f$ is predictive (in particular, if it is a cyclic vector of $\cU$), then even
		\[
		\lim_{d\to\infty} A_d(d+1) = A(d+1).
		\]}
	By the Wiener--Khinchin theorem, autocorrelation is related to the notion of power spectrum through Fourier transformation~\cite{MeGo07}.
	Since Proposition~\ref{prop:autocorr} does not indicate anything about autocorrelation beyond a finite horizon, statements about the reconstruction of the power spectrum seem to be elusive.
	Methods more powerful than the least squares filter can approximate the power spectrum, e.g.~\cite[\S4.5]{colbrook2023residual}.
	\revision{In the next section we consider spectral approximation of~$\cU$. We note that the rightmost term in~\eqref{eq:UdMinusU_bound} will reappear in the pseudospectral bound given there.}
	
	\subsection{Some spectral properties}
	
	We start with the following trivial observation:
	\begin{proposition}
		\label{prop:projspectrum}
		The spectrum of $U_d$ is contained in the unit complex disk.
	\end{proposition}
	\begin{proof}
		As in the proof of Proposition~\ref{prop:autocorr}: Since $\cU$ is unitary and $\cU_d$ is the orthogonal projection of $\cU$, by the non-ex\-pan\-sive\-ness of orthogonal projections we have that $\|\cU_d\| \le 1$, implying the claim.
	\end{proof}
	
	Under additional assumptions, stronger statements about the position and the distribution of the spectrum of $U_d$ can be found in~\cite[Corollaries 2 and~3]{korda2020data}.
	
	How close is the spectrum of $\cU_d$ to that of~$\cU$? The theory of pseudospectra~\cite{TrEm05} for normal operators offers a way to approach this question. For simplicity of notation, we use $\| \cdot \| = \|\cdot\|_{L^2_\mu}$ for the remainder of this section.
	A scalar $\lambda \in \C$ is said to be in the \emph{$\ep$-pseudospectrum} of $\cU$ if there is a $h\in L^2_{\mu}$ with $\|h\|=1$ such that~$\| \cU h- \lambda h\| < \ep$. Since $\cU$ is unitary, hence normal, every value in the $\ep$-pseudospectrum is at most distance $\ep$ away from a true spectral value~\cite{TrEm05}. 
	
	One can now take an eigenpair $(\lambda_d,\phi_d)$ of $\cU_d$ with $\|\phi_d\|=1$ and estimate~$\|\cU \phi_d - \lambda_d \phi_d\|$ to assess how far is $\lambda_d$ from the spectrum of~$\cU$. \revision{Recall the notation introduced prior to Proposition~\ref{prop:autocorr}. With it, we have:}
	\begin{proposition}
		\label{prop:pseudospec}
		Let $\lambda_d$ be an eigenvalue of $\cU_d$. Then $\lambda_d$ is contained in the $\ep$-pseudo\-spectrum of~$\cU$, where
		\begin{equation}
			\label{eq:epsi_pseudo}
			\ep = \frac{\| \Pi_{\mathcal{K}_d}^{\perp} \cU f\|}{\| \Pi_{\mathcal{K}_d^-}^{\perp} f\|}.
		\end{equation}
	\end{proposition}
	\begin{proof}
		Let $(\lambda_d,\phi_d)$ be an eigenpair of $\cU_d$ with $\|\phi_d\|=1$. Note that we seek $\ep = \|\cU \phi_d - \lambda_d \phi_d\|$. Let us write $\phi_d = a f^- + b f$ with $f^- \in \mathcal{K}_d^-$ and~$a,b \in \C$. Then, since $\lambda_d\phi_d = \cU_d \phi_d$, we have that
		\[
		\ep = \| (\cU - \Pi_{\mathcal{K}_d}\cU)\phi_d \| = \| \Pi_{\mathcal{K}_d}^{\perp} \cU (b f) \|,
		\]
		because~$\Pi_{\mathcal{K}_d}^{\perp} \cU f^- = 0$. We thus need to bound $|b|$ from above, given~$\|a f^- + b f\|=1$. This can be done by basic trigonometry. Let us consider the triangle with corners $0,a f^-$, and $\phi_d$, which then have the opposite ``sides'' of lengths $\|bf\|$, $\|\phi_d\|=1$, and $\|a f^-\|$, respectively. Denote the angle at the corner $0$ by $\beta$ and the one at $af^-$ by~$\theta$. It is then immediate that
		\[
		\sin \theta = \frac{\|\Pi_{f^-}^{\perp} \phi_d\|}{\|b f\|}  \ge \frac{\|\Pi_{\mathcal{K}_d^-}^{\perp} \phi_d\|}{\|b f\|} = \frac{\|\Pi_{\mathcal{K}_d^-}^{\perp} f\|}{\|f\|},
		\]
		where we used $\mathrm{span}\{f^-\} \subset \mathcal{K}_d^-$ for the inequality and~$\Pi_{\mathcal{K}_d}^{\perp} f^- = 0$ for the equality. The sine law for this triangle yields
		\[
		\frac{\sin \theta}{\|\phi_d\|} = \frac{\sin \beta}{\|bf\|},
		\]
		giving $|b| = \frac{\sin\beta}{\|f\| \sin\theta}$. Since $\sin\beta\le 1$, combining this with the above proves the claim.
	\end{proof}
	\begin{remark}
		\quad
		\begin{enumerate}[(a)]
			\item  Note that in the expression \eqref{eq:epsi_pseudo} for $\ep$, the denominator represents the error to approximate $f$ by the $(d-1)$ observables $f\circ T^{1-d},\ldots,f \circ T^{-1}$ preceding it, and the numerator represents the error to approximate $\cU f$ by the $d$ observables~$f\circ T^{1-d},\ldots, f$ preceding it. 
			\item In \cite[Lemma 2]{mezic2022numerical}, there was an attempt to make a similar pseudospectral statement.
			The result, however, is questionable, since its proof seems to be incorrect: Their pseudo-eigenfunction $\mathbf{\tilde{e} \cdot \tilde{f}}$ does not have $L^2_{\mu}$-norm equal to~1, in general.
			This also impairs~\cite[Theorem~5]{mezic2022numerical}.
		\end{enumerate}
	\end{remark}
	Despite the appealing form of \eqref{eq:epsi_pseudo}, we do not seem to be able to draw practical conclusions from it for~$d\to \infty$. With going to a subsequence, however, we can make connections to the point spectrum. 
	The following is a consequence of~\cite[Theorem~4]{korda2018convergence}, by noting that $\cU_d$ is the $L^2_{\mu}$-orthogonal projection on~$\cK_d$.
	
	\begin{proposition}
		\label{prop:approxpoint}
		Let $(\lambda_d,\phi_d)$ be eigenpairs of $\cU_d$ for $d\in \N$ with~$\|\phi_d\|=1$. If $f$ is a cyclic vector of $\cU$, then there is a subsequence $(d_k)_k \subset \N$ such that $\phi_{d_k} \stackrel{w}{\to} \phi$ and $\lambda_{d_k} \to \lambda$, and~$\cU \phi = \lambda \phi$. In particular, if $\phi\neq 0$, then $(\lambda,\phi)$ is an eigenpair of~$\cU$.
	\end{proposition}

	At this point it is important to note that other methods than the least squares filter allow to make stronger claims about faithful approximation of (parts of) the spectrum of~$\cU$, see~\cite{GiDa19,korda2020data,giannakis2021delay,colbrook2021rigorous,colbrook2023mpedmd,valva2023consistent}.

	\subsection{Approximate filter}

	Both Theorem~\ref{thm:DensePrevalent} and Proposition~\ref{prop:DenseGd} imply that the set of systems that are predictive for a large set of observables is weakly dense in~$\mpt$. As the weak topology is metrizable \cite[p.~64]{Hal17}, there is some predictive map ``arbitrary close'' to our original system~$T$. Recall, however, that a practical interpretation of this metric seems hardly possible according to Halmos.
	To remedy this fact somewhat, we will consider this closeness statement on the level of filters. In particular, we show in Proposition~\ref{prop:filterapprox} below that there is a predictive system that produces a filter arbitrary close to that produced by~$T$.
	
	Note that, by ergodicity of $T$, for almost every $x\in\set{X}$ we obtain the same filter~$c$ (cf.\ Lemma~\ref{lem:filter_convergence}). A quick inspection of the structure of the objective function~$J$ in~\eqref{eq:filter_new} reveals that the actual observation sequence is of little relevance for the filter $c$; what matters are the correlation integrals $\int (f\circ T^{-i})\, (f\circ T^{-j})\,d\mu$, $i,j=-1,0,\ldots,d-1$. Note that Lemma~\ref{lem:weakapprox}(5) gives for a sequence $(T_n)_n \subset\mpt$ with $\smash{ T_n\stackrel{w}{\to} T }$ that
	\begin{equation}
		\label{eq:weakapprox_corr}
		\lim_{n\to\infty} \int (f\circ T_n^{-i})\, (f\circ T_n^{-j})\,d\mu = \int (f\circ T^{-i})\, (f\circ T^{-j})\,d\mu	
	\end{equation}
	for all~$i,j = -1,0,\ldots,d-1$. Since for a fixed $d\in\N$ there are finitely many pairs $(i,j)$ to consider, this convergence is uniform across these pairs. Thus, Theorem~\ref{thm:DensePrevalent}, Equation~\eqref{eq:weakapprox_corr}, together with 
	Lemma~\ref{lem:filter_convergence} gives
	\begin{proposition}
		\label{prop:filterapprox}
		Let $T\in\mpt$ and $d\in\N$ a fixed delay depth. Let $f\in L^2_\mu$ be an observable such that $\mathcal{B}_d$ is a basis. Then there are predictive $S\in\mpt$ that produce filters arbitrary close to that of~$T$.
	\end{proposition}
	
	\begin{remark}
		In particular, Proposition~\ref{prop:filterapprox} implies that based on the filter alone, which is computed from (finite or infinite) observation data, one is not able to infer whether the system at hand is predictive. This also implies that the spectrum of $\cU_d$ (which is for a fixed $d$ a continuous function of the filter vector $c$) is not able to distinguish whether the observed time series comes from a predictive system.
	\end{remark}
	
	We stress that Proposition~\ref{prop:filterapprox} is a statement for a fixed delay depth, and in particular does not imply anything about the forecast error~$\| \cU_d f - \cU f\|_{L^2_\mu}$ in the limit~$d\to \infty$.

	\section{Numerical examples}
	\label{sec:examples}
	
	We illustrate our results with a variety of examples properties of the linear least squares filter. In particular, we will consider the prediction error and the autocorrelations.
	
	\subsection{Numerical setup}
	
	The filter is learned on (observations of) a long ``training'' trajectory of length~$(m+1)$, cf.~\eqref{eq:filter_finite}. The mean squared forecast error is computed on a different ``testing'' trajectory, of the same length $(m+1)$, by sliding the filter along the trajectory and comparing the prediction with the next (the true) element on the trajectory. By ergodicity, the mean square error thus approximates~$\smash{ \| \cU_d f - \cU f \|_{L^2_\mu}^2 }$ in the limit of~$m\to\infty$.
	
	The autocorrelations are compared as follows. We generate a large sample of $\mu$-distributed initial points~$x^{(i)}_1$, $i=1,\ldots,N$.
	This is done either by random uniform sampling, as in most cases it is the unique ergodic distribution, or by a long trajectory of a rationally independent stepsize compared with the one used for learning and testing, in the case of the one time-continuous system. From each of these initial conditions we generate trajectory ``snippets'' of length $d$, denoted by $x^{(i)}_1, x^{(i)}_2 = T x^{(i)}_1, \ldots, x^{(i)}_d = T^{d-1} x^{(i)}_1$, and continue each snippet in two ways:
	\begin{enumerate}[(i)]
		\item By the true map, and apply the observation map to the elements at the end, giving
		\[
		f \big(x^{(i)}_1\big), f \big( x^{(i)}_2\big), \ldots, f \big( x^{(i)}_d\big), f \big( x^{(i)}_{d+1} = T x^{(i)}_d\big), f \big( x^{(i)}_{d+2} = T^2 x^{(i)}_d\big), \ldots
		\]
		\item By applying the filter iteratively to the observed trajectory ``snippet'' to generate an observation sequence of the same length as in~(i), giving
		\[
		z^{(i)}_1 = f \big(x^{(i)}_1\big), z^{(i)}_2 = f \big( x^{(i)}_2\big), \ldots, z^{(i)}_d = f \big( x^{(i)}_d\big), z^{(i)}_{d+1}, z^{(i)}_{d+2}, \ldots,
		\]
		where $z^{(i)}_n$, $n>d$, is obtained by the filter applied to~$z^{(i)}_{n-d},\ldots, z^{(i)}_{n-1}$.
	\end{enumerate}
	We then approximate the autocorrelations of the system and the filter by
	\begin{align*}
		A(n) & = \frac1N \sum_{i=1}^N f\big( x^{(i)}_d \big) f\big( x^{(i)}_{d+n} \big), \\
		A_d(n) & = \frac1N \sum_{i=1}^N z^{(i)}_d z^{(i)}_{d+n},
	\end{align*}
	respectively. Unless stated differently, we use~$m=N=10^4$ in the examples.
	
	\subsection{Ergodic torus rotation}
	\label{ssec:ergodictorus}
	
	First we consider the ergodic rotation of the unit 2-torus~$S^1\times S^1$:
	\[
	T(x) = \begin{pmatrix}
		x_1 + \sqrt{2} \\
		x_2 + \sqrt{3}
	\end{pmatrix} \mod 1.
	\]
	We use the observables $f_1(x) = \exp( \sin(4\pi x_1) + \cos(6\pi x_2 ))$ and $f_2 = \chi_{[0,1/2]\times[1/2,1]}$, where $\chi_A$ denotes the characteristic function of the set~$A$.
	
	\begin{figure}[htbp]
		\centering
		\includegraphics[width = 0.51\textwidth]{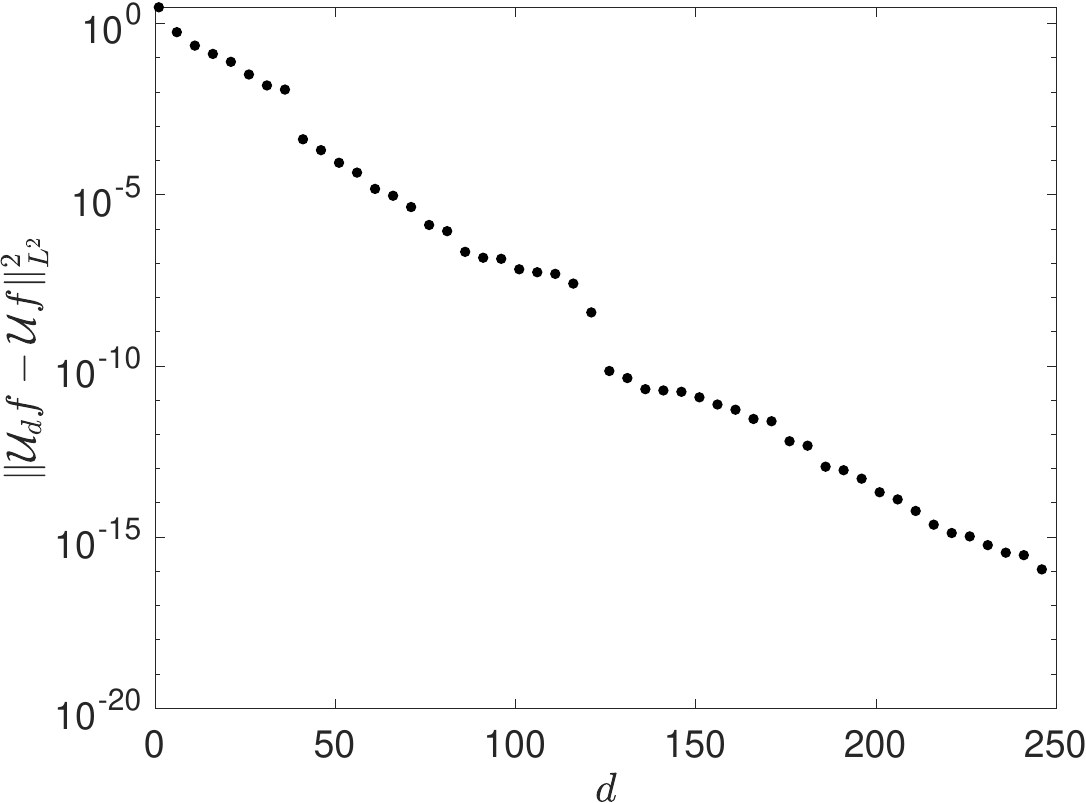}
		\hfill
		\includegraphics[width = 0.48\textwidth]{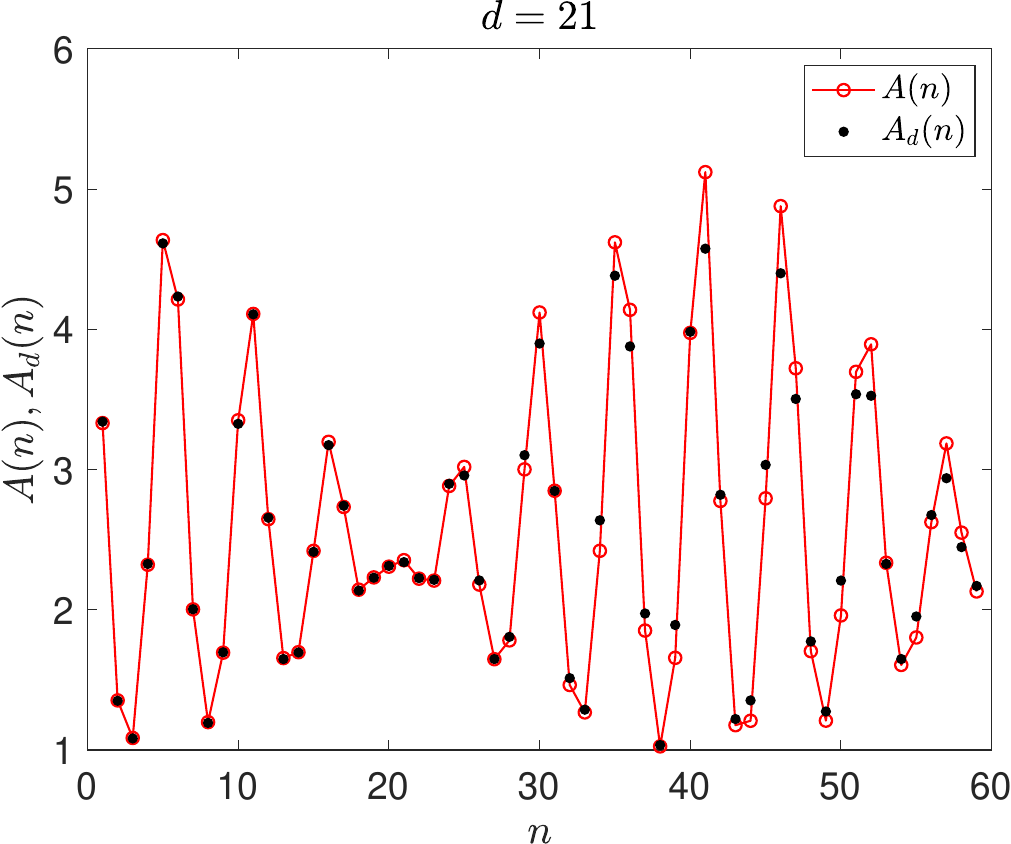}
		\caption{Ergodic torus rotation, results for observable~$f_1$. Left: $L^2$ prediction error. Right: autocorrelations.}
		\label{fig:torus}
	\end{figure}
	
	\begin{figure}[htbp]
		\centering
		\includegraphics[width = 0.49\textwidth]{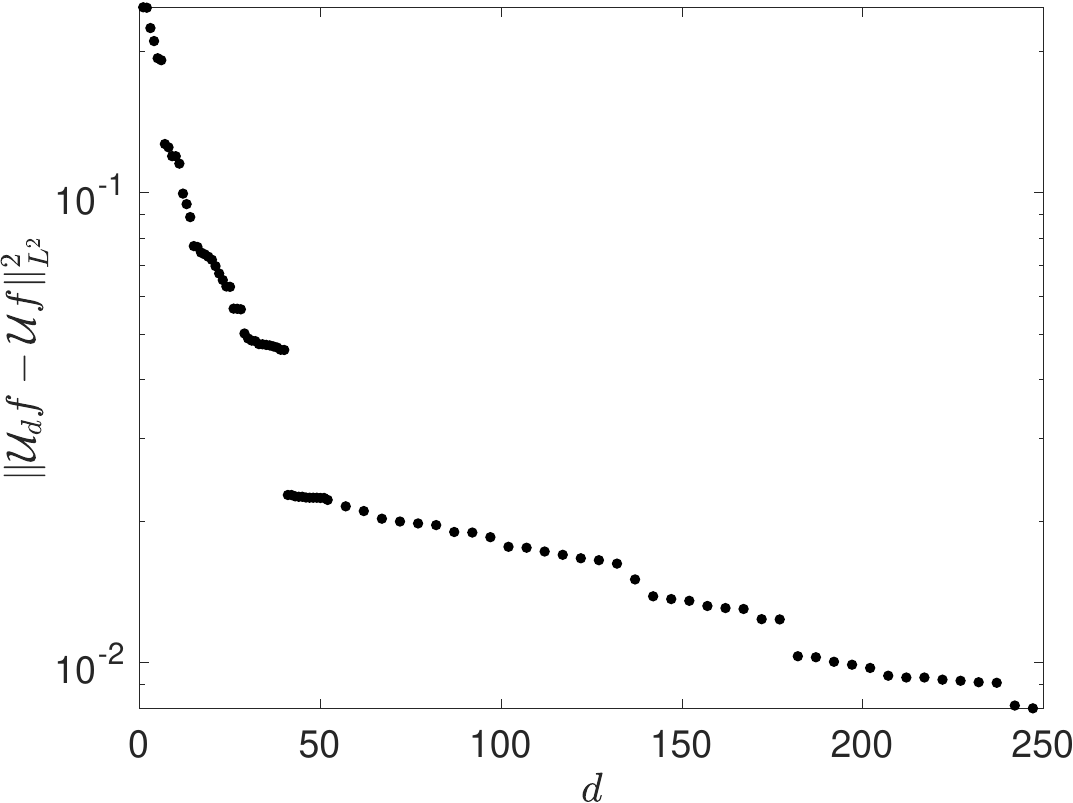}
		\caption{Ergodic torus rotation. $L^2$ prediction error for observable $f_2$.}
		\label{fig:torus_obs2}
	\end{figure}
	
	In the mean squared forecast error we observe exponential convergence in the number of delays, see Figure~\ref{fig:torus} (left) and Figure~\ref{fig:torus_obs2}. 
	Note the scale on the vertical axes: The convergence is slower the rougher the observable is, and in particular for non-continuous functions. Why in Figure~\ref{fig:torus_obs2} around $d=35$ the sudden slow-down occurs is yet to be understood.
	In Figure~\ref{fig:torus} (right) we see that the filter's autocorrelations $A_d(n)$ are exact up to the chosen delay depth ($d=21$ there), and qualitatively keep the correct behavior even up to~$n=3d$.
	
	\subsection{The affine twist map}
	
	We consider
	\[
	T(x) = \begin{pmatrix}
		x_1 + \sqrt{2} \\
		x_1 + x_2
	\end{pmatrix} \mod 1,
	\]
	again on the unit 2-torus.
	This map has mixed spectrum: Its Kronecker factor\footnote{The Kronecker factor of a transformation is the maximal factor with pure point spectrum.} is the irrational circle rotation $R_{\sqrt{2}}$, the spectrum
	in the orthogonal complement to this factor is countable Lebesgue \cite[Example 3.17]{KatokThouvenot}: Roughly speaking this part of the spectrum can be represented by the Lebesgue measure on $\bT$ and has everywhere infinite multiplicity.
	In summary, the affine twist map has so-called \emph{quasi-discrete spectrum}.\footnote{A quasi-eigenfunction of order one is an ordinary eigenfunction of~$\cU_T$. A function $f\in L^2$, $f\neq 0$, is called a \emph{quasi-eigenfunction for $T$ of order $n+1$} if $\cU_Tf=\varphi f$, where $\varphi$ is a quasi-eigenfunction of order~$n$. One says that $T$ has a \emph{quasi-discrete spectrum} if the span of quasi-eigenfunctions of all possible orders is dense in $L^2$.}
	
	We use the observable $f(x) = \exp( 2\sin(4\pi x_1) + \cos(6\pi x_2 ))$.
	\begin{figure}[htbp]
		\centering
		\includegraphics[width = 0.49\textwidth]{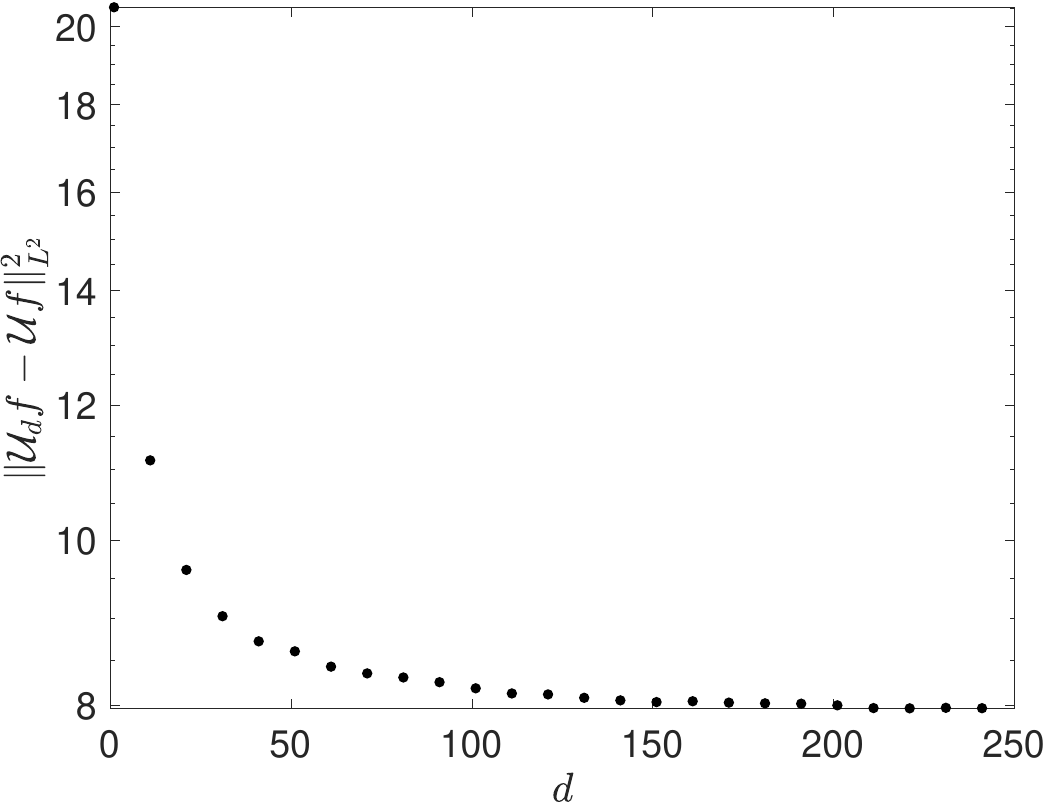}
		\hfill
		\includegraphics[width = 0.49\textwidth]{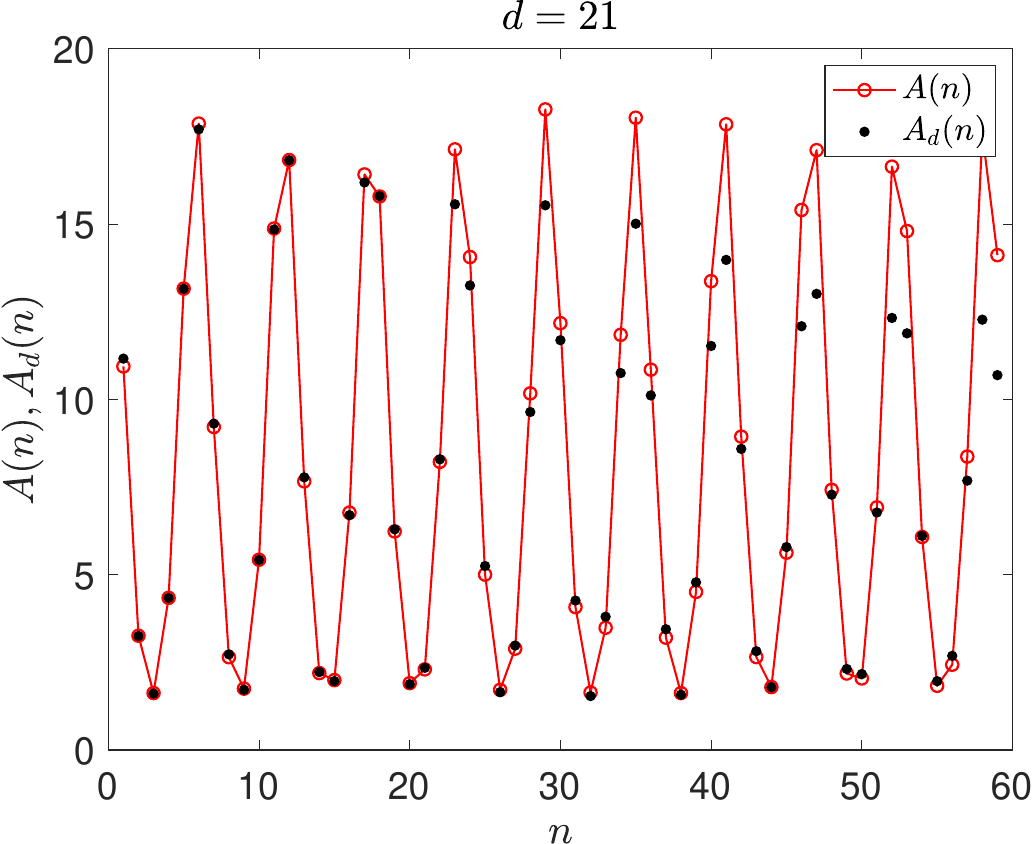}
		\caption{$L^2$ prediction error (left) and autocorrelations for the affine twist map (right).}
		\label{fig:affine}
	\end{figure}
	In Figure~\ref{fig:affine} (left), we observe that at some point the prediction error ceases to improve with increasing delay. This indicates that $(T,f)$ is not predictive (not even weakly).
	
	Unsurprisingly, reducing the dependence of the observable $f$ on the ``mixing coordinate''~$x_2$ improves the prediction error. For $f(x) = \exp( 2\sin(4\pi x_1) + 0.01\cos(6\pi x_2 ))$ the prediction error curve in Figure~\ref{fig:affine} (left) looks similar in shape, but converges to a nonzero value less than~$10^{-3}$ (not shown). For observables satisfying $f(x) = \tilde{f}(x_1)$ for some $\tilde{f}$ the prediction error behaves as in the previous example.
	
	The autocorrelations behave as predicted by Proposition~\ref{prop:autocorr} (we chose $d=21$), they seem to maintain the right ``frequency'' even for $n>d$, however, their ``amplitudes'' decay faster.
	
	\subsection{Von Neumann--Kakutani's transformation} \label{subsec:odometer}
	
	The von Neumann--Kakutani's transformation~\cite{neumann1932operatorenmethode}, $T:[0,1] \to [0,1]$, is also known under the names dyadic odometer or Corput's transformation.
	In \cite[section 1.3.1]{ferenczi1997systems}, it is given in the form
	\[
	T\Big(1-\frac{1}{2^n} + y \Big) = \frac{1}{2^{n+1}}+y \quad \forall\ 0\le y < \frac{1}{2^{n+1}}, \ n\in\N_0.
	\]
	After some manipulations, the following form can be obtained:
	\[
	T(x) = x-1 + \frac{3}{2^{n(x)}},\qquad n(x):=  \lfloor -\log_2(1-x) \rfloor.
	\]
	This transformation is rank one and has discrete spectrum, with the eigenvalues being the dyadic rationals~\cite{ferenczi1997systems}. The main difference to the torus rotation (which also has discrete spectrum) is that while the spectrum of the latter in section~\ref{ssec:ergodictorus} is generated by the two irrational factors~$\sqrt{2}$ and~$\sqrt{3}$, here there are countably many numbers generating the spectrum.
	We use the observable~$f(x) = \exp(\sin(6\pi x))$, and~$N=10^5$ for increased testing accuracy.
	
	Although we observe monotonic decrease in the prediction error in Figure~\ref{fig:odometer} (left), the error decrease is slowing down for larger delay depths. We attribute this to the complexity of the system's spectrum, of which the method is picking up more and more as $d$ increases (see also Figure~\ref{fig:odometer_spectrum}). The autocorrelations also show a more complex picture as compared with the previous examples, thus we depict it for a larger delay depth $d=51$ in Figure~\ref{fig:odometer} (right). Despite the complexity, we note that the autocorrelations are very well reproduced on the considered time scale~$n=3d$.
	
	We show the spectra in Figure~\ref{fig:odometer_spectrum}, which are in this case also reproduced by the least squares filter to a high accuracy.
	
	\begin{figure}[htbp]
		\centering
		\includegraphics[width = 0.49\textwidth]{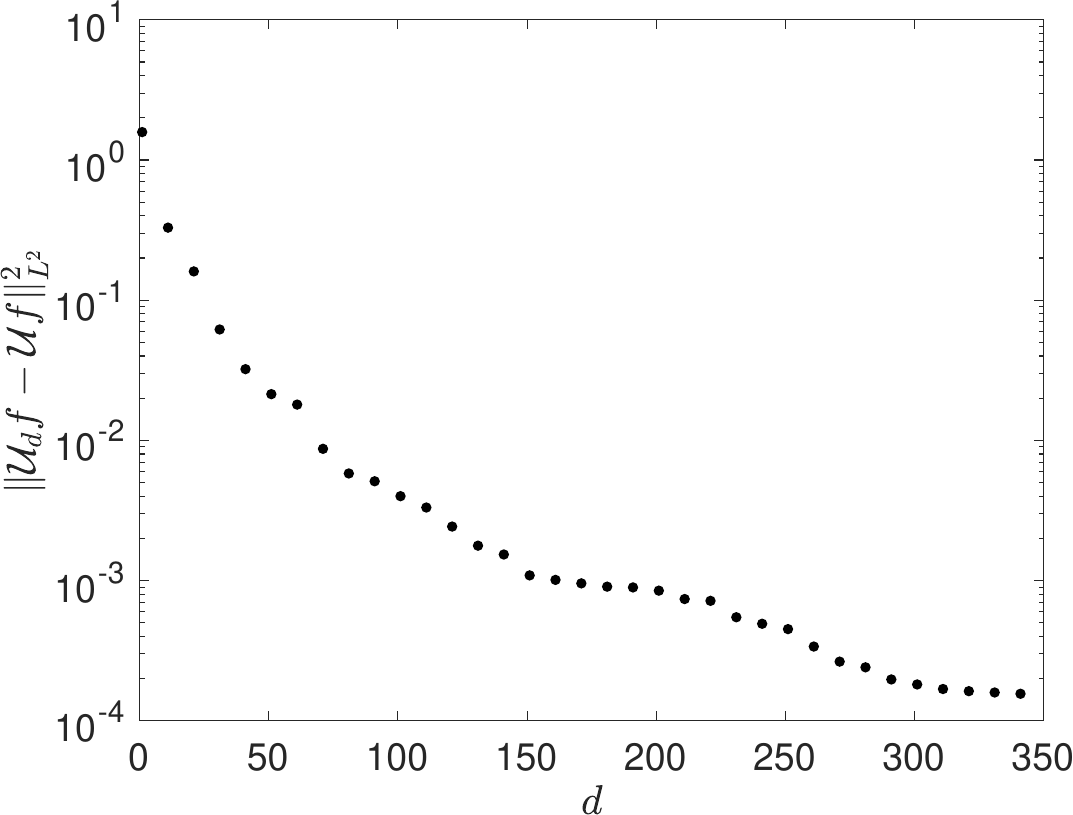}
		\hfill
		\includegraphics[width = 0.49\textwidth]{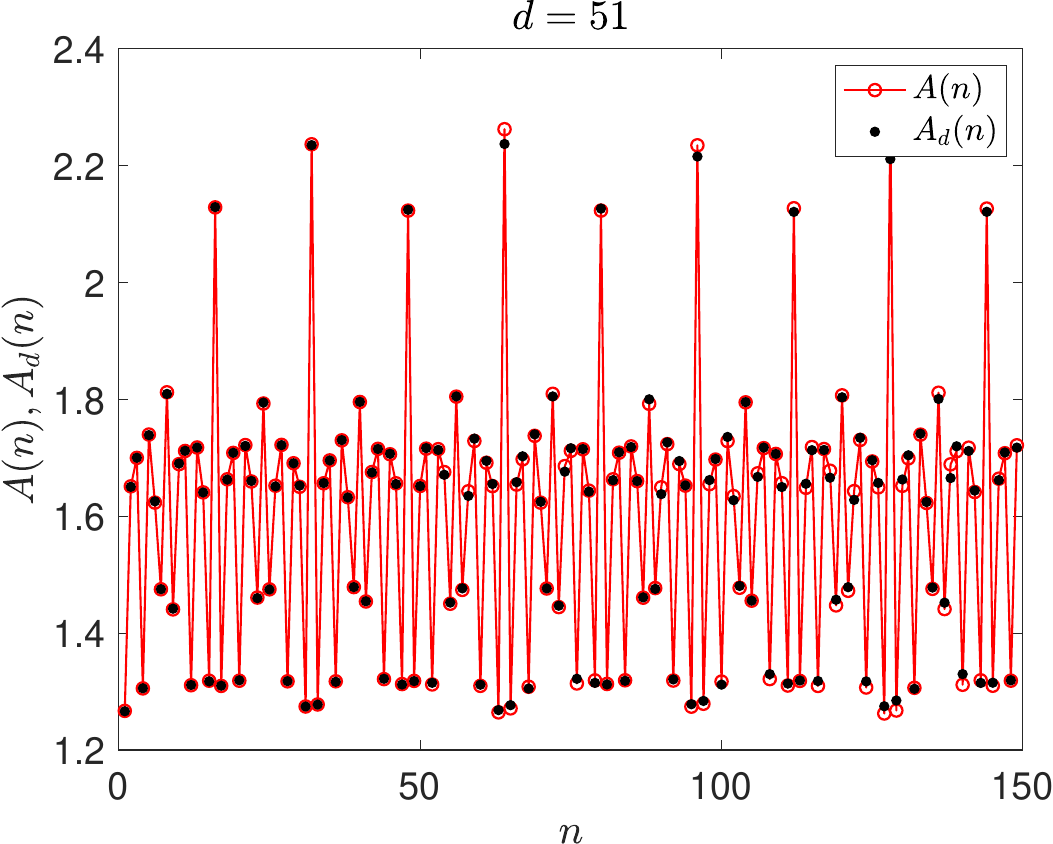}
		\caption{The von Neumann--Kakutani's transformation. $L^2$ prediction error (left) and autocorrelations (right).}
		\label{fig:odometer}
	\end{figure}
	
	\begin{figure}[htbp]
		\centering
		\includegraphics[width = 0.49\textwidth]{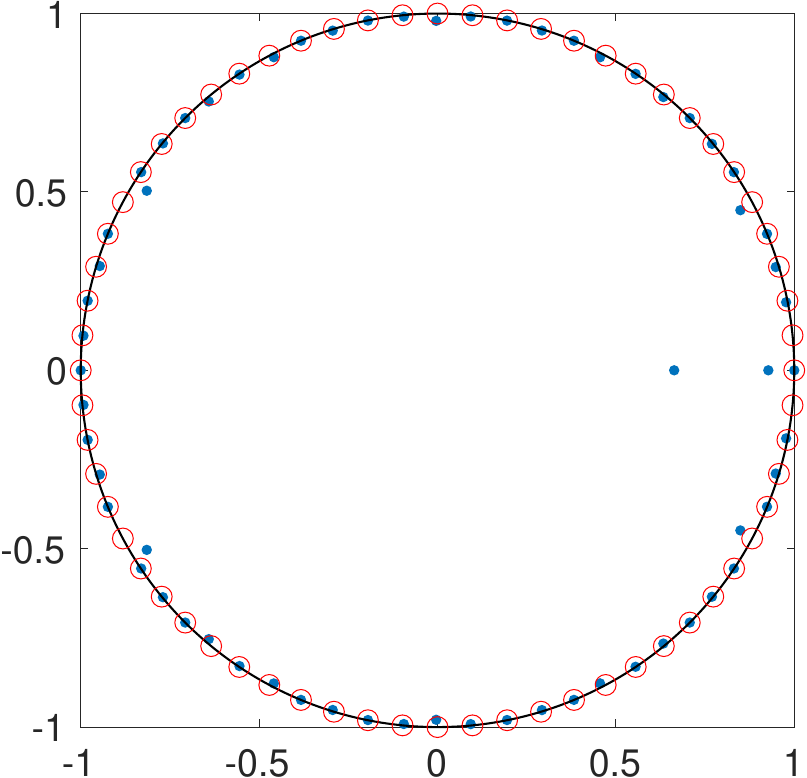}
		\hfill
		\includegraphics[width = 0.49\textwidth]{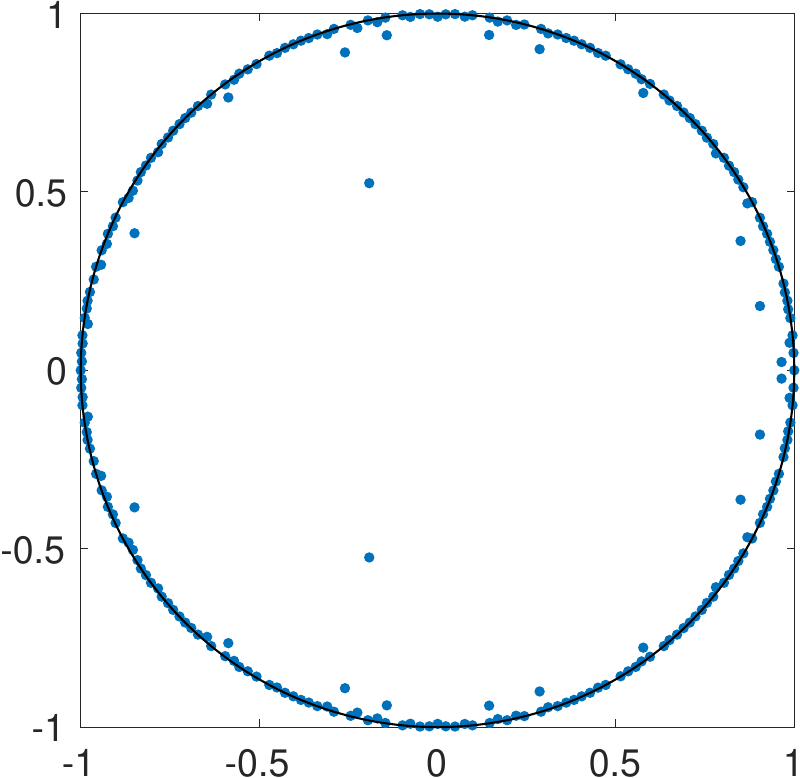}
		\caption{Spectra for the von Neumann--Kakutani's transformation. Blue dots indicate the spectrum of $U_d$ and the black line is the complex unit circle. Left: $d=64$, red circles show ``dyadic rationals'' $\exp(2\pi i \frac{k}{64})$, $k=0,\ldots,63$, that are known to be eigenvalues of~$\cU$. Right:~$d=250$.}
		\label{fig:odometer_spectrum}
	\end{figure}

	\subsection{The Lorenz '63 model}
	
	We consider the Lorenz '63 (L63) system
	\[
	x_1' = \sigma(x_2-x_1), \qquad 
	x_2' = x_1(\rho-x_3) - x_2, \qquad
	x_3' = x_1 x_2 - \beta x_3,
	\]
	with the usual choice of parameters~$\sigma=10, \beta = \frac83, \rho=28$. The discrete time system $T$ that generates the time series is taken as the time-$t$-flow map $\phi^t$ of this system for~$t=0.05$. It is realized by the ``classical'' 4th order Runge--Kutta method with step size~$5\cdot 10^{-4}$.
	\revision{This flow time is considerably smaller than the Lyapunov timescale---the inverse of the largest Lyapunov exponent, used to measure the timescale on which chaotic motion can be observed---of the L63 system, known to be approximately~$t_{\mathrm{Lyap}}\approx 1.1$ \cite{geurts2020lyapunov}.}
	
	\revision{Two equilibria of the L63 system are at $x_{\pm} = (\pm\sqrt{\beta(\rho-1)}, \pm\sqrt{\beta(\rho-1)}, \rho-1 )$. We set $\delta = \sqrt{\beta(\rho-1)}$ and consider the two observables
		\[
		f_1(x) = x_1 \quad \text{and} \quad f_2(x) = \exp\left( - \frac{\|x-x_+\|_2^2}{(\delta/3)^2} \right),
		\]
		with $f_2$ depicted in Figure~\ref{fig:l63_obs}.}
	
	\begin{figure}[htbp]
		\centering
		\includegraphics[width=0.3\textwidth]{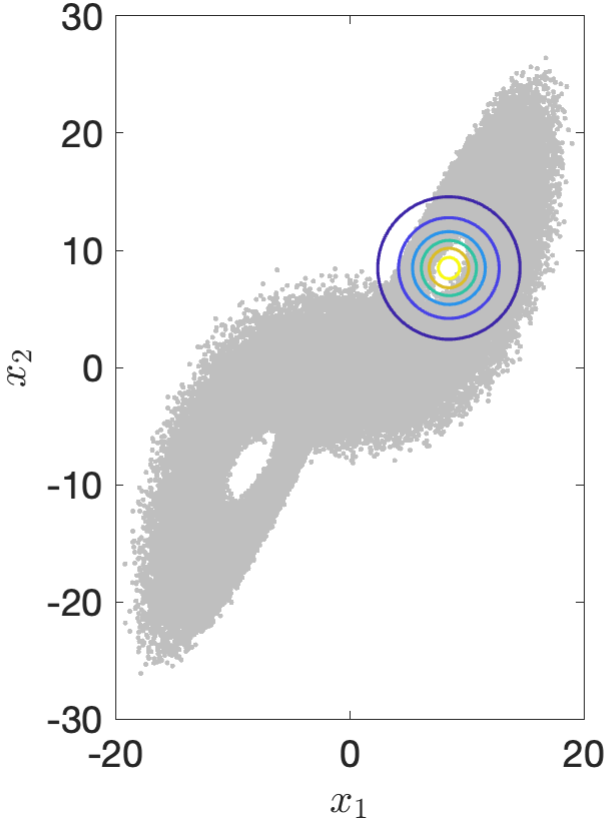}
		\caption{Observable for the L63 system, contours shown at~$z=27$. Grey dots indicate the projections of the training points onto the $(x_1,x_2)$ plane.}
		\label{fig:l63_obs}
	\end{figure}
	
	The initial points $x^{(i)}_1$ for the autocorrelation computation are obtained from a trajectory of length $N=10^6$ generated by the time-$t$-flow map of L63 with~$t=40 h$, using the step size~$h=5.3\cdot 10^{-4}$ and a random initial point close to~$(4,7,16)^{\intercal}$. \revision{We choose~$m=10^5$, as we have validated separately that increasing $m$ does not increase the fidelity of our results.}
	
	The Lorenz system is mixing~\cite{luzzatto2005lorenz} and hence has purely continuous spectrum.
	\revision{Theorem~\ref{thm:weakpred} tells us to expect predictiveness only if the trace measure satisfies the Szeg{\H o} condition. For the observable $f_1$ the trace measure, and in particular the associated density $w$, has been computed in~\cite[Figure~11 (top left)]{korda2020data}, showing that $w$ is bounded away from zero\footnote{\label{foot:density}In \cite{korda2020data}, the flow time $t=0.2$ is used, whereas we took~$t=0.05$. We used their method to confirm that both spectral densities of $\phi^{0.05}$, with respect to $f_1$ and $f_2$, are bounded away from zero (not shown).}. This implies that the Szeg{\H o} condition is not satisfied and that the prediction error cannot converge to zero. Indeed, we observe that for increasing $d$ the prediction error saturates at a nonzero value, see Figure~\ref{fig:l63_x1} (left). The left panel actually shows a comparison of multistep prediction errors of filters for several delay depths $d$ and flow times~$t \in \{0.05, 0.1, 0.2, 0.4\}$ (blue, red, orange, and purple, respectively). The prediction was made for time $0.4$, i.e., for $8,4,2,1$ steps of the filters associated to the respective flow times. The corresponding errors are shown versus the product $dt$ of delay depth multiplied by the flow time. We observe that the error does not improve after $d t \approx 5$, i.e., the filter gains no advantage from accessing observations from more than a handful Lyapunov times in the past. Additionally, the (in $d$ saturated) errors $\smash{ \| \cU^{0.4} f - \cU^{0.4}_d f \|^2_{L^2_\mu} }$ increase with increasing~$t$. This is in accordance with the intuition that ``larger flow times lose more information about the initial condition''. An interesting question, that we do not pursue any further here,  would be how the error behaves for $t\to 0$ as $dt$ is kept constant.
	We also note that the blue and red curves ($t=0.05$ and $t=0.1$, respectively) are not monotonic in~$d$. This is not in conflict with the proof of Proposition~\ref{prop:forecasterror}, because there monotony is shown only for the one-step forecast error, and here we depict multistep errors for these flow times.}
	
	\revision{Figure~\ref{fig:l63_x1} (right) shows the autocorrelation comparison as in the previous examples, for $t=0.05$ and~$d=21$. For steps beyond the delay depth the filter underestimates the true autocorrelation.}
	
	\begin{figure}[htb]
		\centering
		\includegraphics[width = 0.49\textwidth]{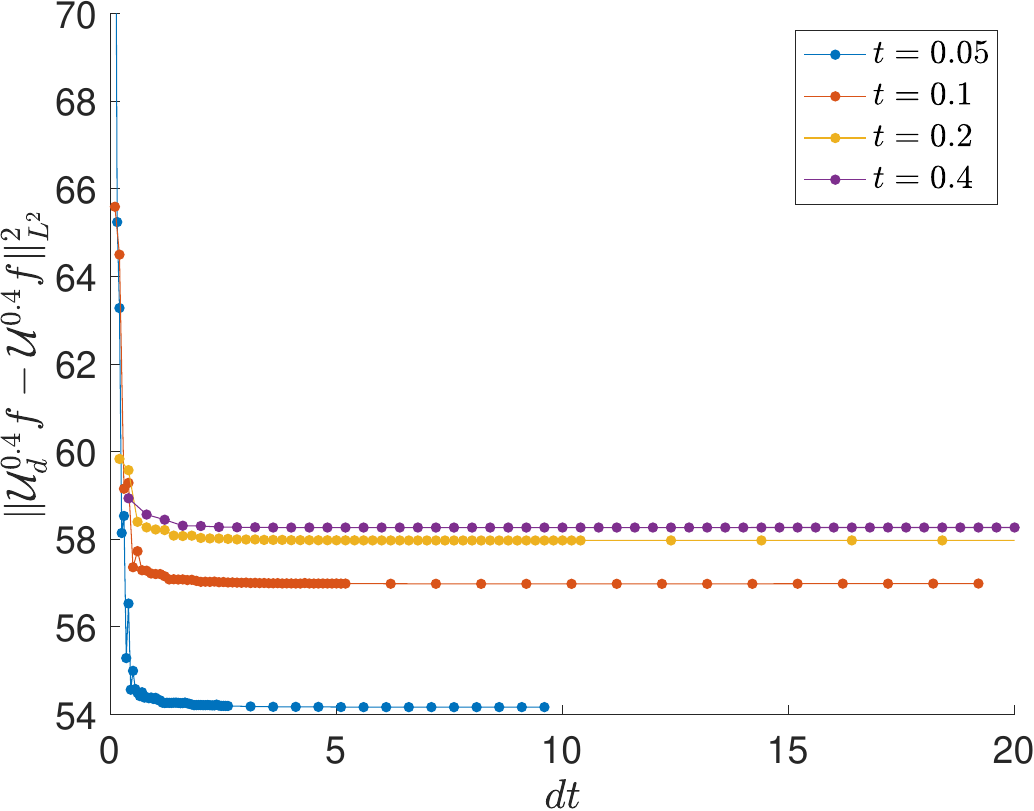}
		\hfill
		\includegraphics[width = 0.49\textwidth]{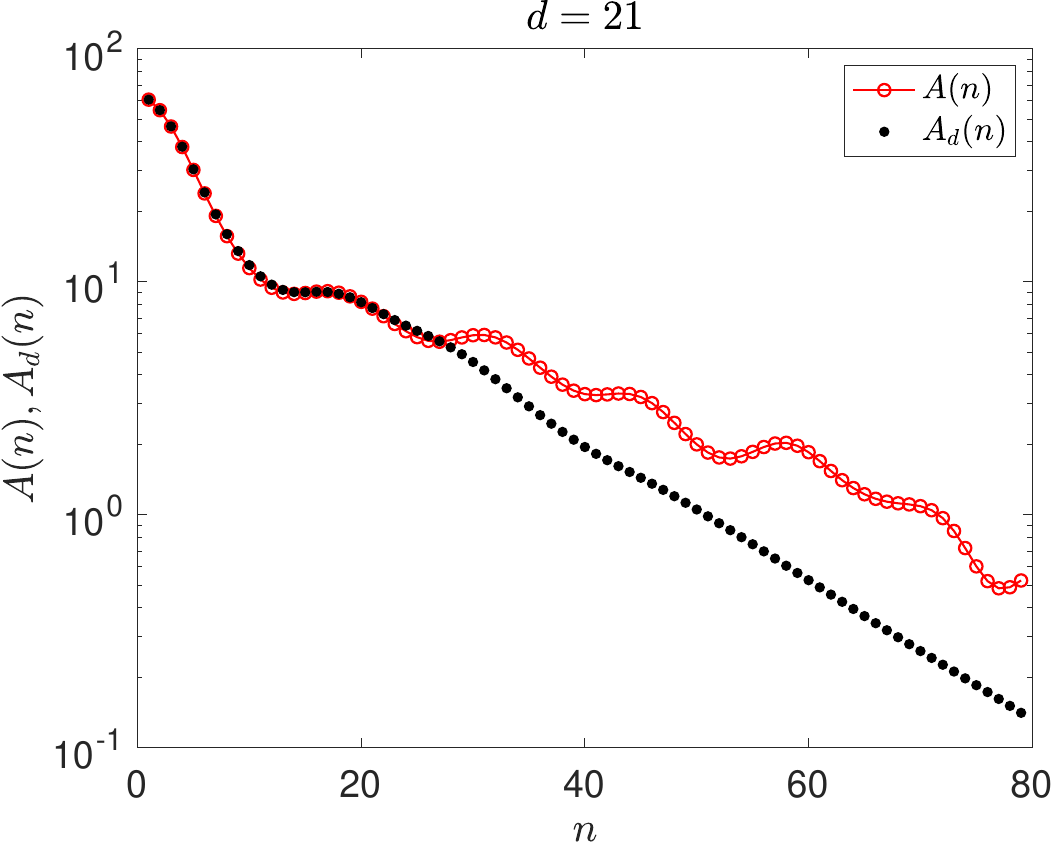}
		\caption{\revision{L63 system with observable $f_1$. Left: $L^2$ prediction error for prediction time~$0.4$, using flow times~$t \in \{0.05, 0.1, 0.2, 0.4\}$. Right: Autocorrelation for the filter with $d=21$ using flow time~$t=0.05$.}}
		\label{fig:l63_x1}
	\end{figure}
	
	\revision{The observable $f_2$ is chosen to pick up a finer-scale dynamics around the equilibrium~$x_+$. Note that by Footnote~\ref{foot:density}, we expect it to be unpredictive as well. Indeed, this is confirmed by the error saturating at a nonzero value in Figure~\ref{fig:l63} (left).} The autocorrelations up to the delay depth $d=21$ are reproduced exactly, and their exponential decay and qualitative behavior persist beyond that lag; shown in Figure~\ref{fig:l63} (right) up to~$n=80 \approx 4d$.
	
	\begin{figure}[htbp]
		\centering
		\includegraphics[width = 0.49\textwidth]{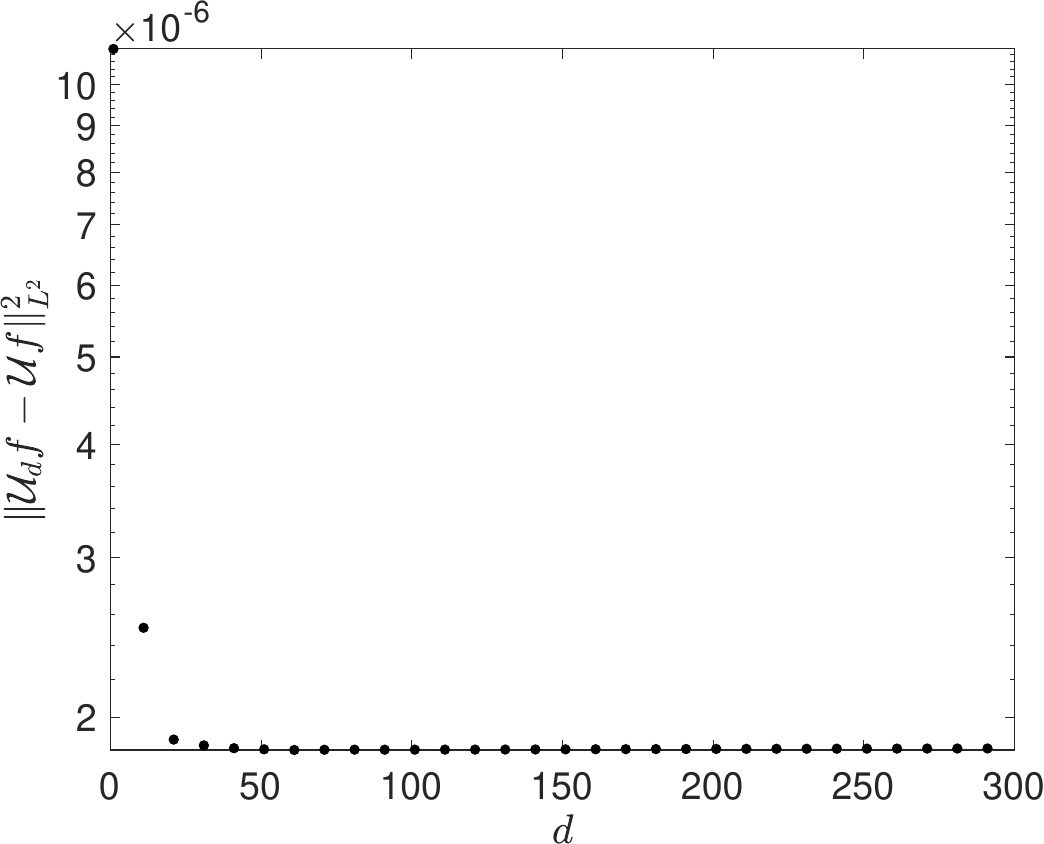}
		\hfill
		\includegraphics[width = 0.49\textwidth]{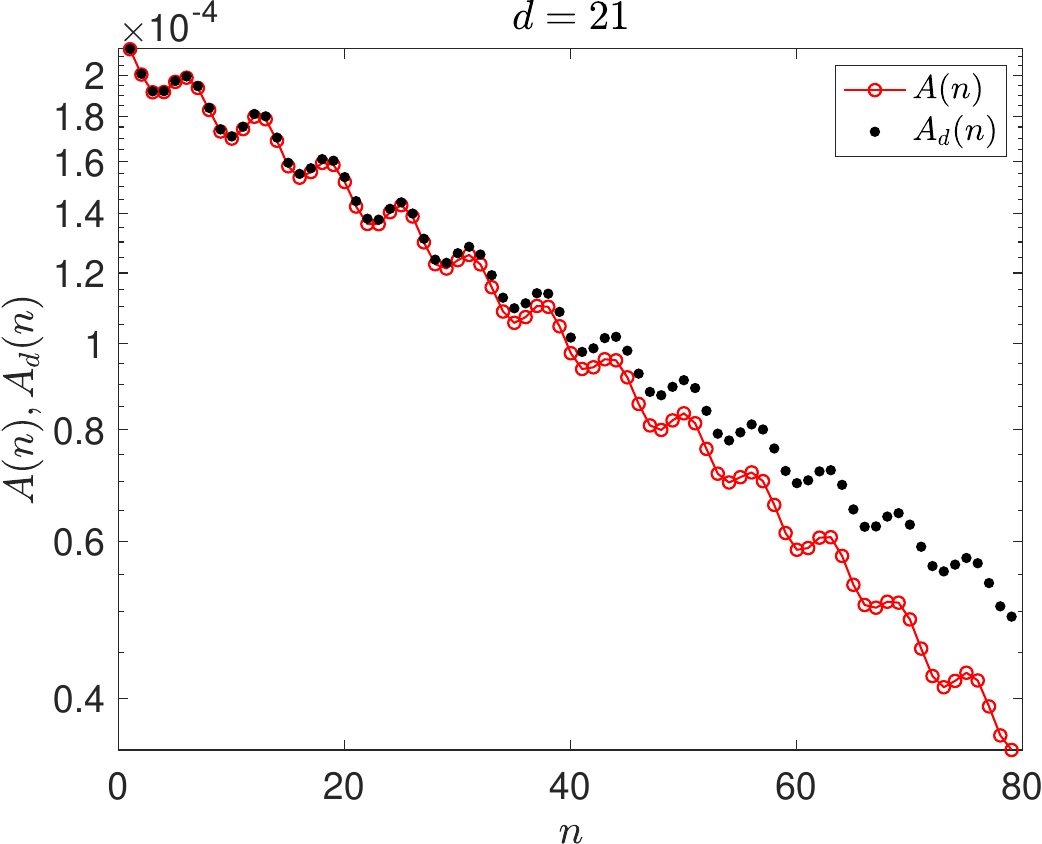}
		\caption{L63 system with observable $f_2$. $L^2$ prediction error (left) and autocorrelations (right).}
		\label{fig:l63}
	\end{figure}

	\section{Discussion and outlook}
	\label{sec:outlook}
	
	We considered the linear least squares filter (Wiener filter) for prediction of time series and showed Takens-type theorems about the size of the set of dynamics-observables pairs that allow asymptotically exact prediction in the limit of infinite delay depth. Since our setup was ergodic-theoretic and a central role was played by the Koopman operator, we call such a result a ``Koopman--Takens'' theorem. The main results, to be found in sections~\ref{ssec:ST} and~\ref{sec:genericity}, show among other things that a generic measure-preserving transformation together with a generic square-integrable observable is a predictive pair.
	
	Our results showed that systems with discrete spectrum are particularly well suited for the prediction problem (see, e.g., Proposition~\ref{prop:DiscSpecWeakPredict} and Theorem~\ref{thm:DensePrevalent}); this has also been observed by others, see, e.g.~\cite{GiDa19}. In general, we conjectured that discrete spectrum together with ergodicity (or perhaps even merely the rank one property) implies that a prevalent set of cyclic observables exist, see Conjecture~\ref{conj:discr_spec_prevalence}.
	Interestingly, however, mixing systems (that are known to have purely continuous spectrum) can be predictive as well, see the discussion after Proposition~\ref{prop:DiscSpecWeakPredict} and Remark~\ref{rem:mixingPred}.
	Further, Proposition~\ref{prop:filterapprox} indicated that the filter with finite delay depth alone does not allow us to draw conclusions about predictiveness. This can be done by validating the mean squared forecast error on a test set.
	In numerical experiments not reported here we observed that the spectrum of $U_d$ can suggest information about the type of spectrum of the system; see also the discussion on the spectral convergence of Hankel DMD in~\cite{korda2020data}. This needs further rigorous elaboration.
	
	On a quantitative level our numerical experiments suggest, on the one hand, that the forecast error decays with increasing delay depth faster for smooth observables than for nonsmooth (e.g., discontinuous) ones. On the other hand, less complicated spectrum has a similar effect: While the torus rotation in section~\ref{ssec:ergodictorus} has a spectrum generated by two scalars (the entries of the translation vector), the von Neumann--Kakutani's transformation has a spectrum generated by all $2^{-n}$, $n\in \N$.
	
	Natural next steps would be to consider vector-valued observations, investigate the effect of noise in the observations, and to derive quantitative error bounds for finite delay depths (in contrast to the asymptotic results given herein). Further, even closer to the spirit of Takens' theorem would be to have results on some measure-theoretic reconstruction of the state space in terms of the infinite observation sequences.

	\section*{Acknowledgments}
	
	\revision{We thank Friedrich Philipp for discussions on weak predictiveness and for a part of the proof of Theorem~\ref{thm:weakpred}. We also thank the referees of the first version of this paper for their careful reading and insightful comments.}
	
	Péter Koltai has been partially supported by Deutsche Forschungsgemeinschaft (DFG) through grant CRC 1114 ``Scaling Cascades in Complex Systems'', Project Number 235221301, Project A08 ``Characterization and prediction of quasi-stationary atmospheric states’’. 
	
	The research of Philipp Kunde is part of the project No.\ 2021/43/P/ST1/02885 co-funded by the National Science Centre and the European Union's Horizon 2020 research and innovation programme under the Marie Sklodowska-Curie grant agreement no.\ 945339. He also thanks Freie Universität Berlin for its hospitality during three research visits related to this paper.

	{
		\small
		\bibliographystyle{myalpha}
		\bibliography{references}

\newcommand{\etalchar}[1]{$^{#1}$}
\begin{thebibliography}{DHVMZ16}

\bibitem[AM17]{arbabi2017ergodic}
H.~Arbabi and I.~Mezi{\'c}.
\newblock Ergodic theory, dynamic mode decomposition, and computation of
  spectral properties of the {K}oopman operator.
\newblock {\em SIAM Journal on Applied Dynamical Systems}, 16(4):2096--2126,
  2017.

\bibitem[AP00]{Alpern}
S.~Alpern and V.~S. Prasad.
\newblock {\em Typical dynamics of volume preserving homeomorphisms}, volume
  139.
\newblock Cambridge University Press, 2000.

\bibitem[BBP{\etalchar{+}}17]{brunton2017chaos}
S.~L. Brunton, B.~W. Brunton, J.~L. Proctor, E.~Kaiser, and J.~N. Kutz.
\newblock Chaos as an intermittently forced linear system.
\newblock {\em Nature communications}, 8(1):19, 2017.

\bibitem[BCKB23]{bakarji2023discovering}
J.~Bakarji, K.~Champion, N.~J. Kutz, and S.~L. Brunton.
\newblock Discovering governing equations from partial measurements with deep
  delay autoencoders.
\newblock {\em Proceedings of the Royal Society A: Mathematical, Physical and
  Engineering Sciences}, 479(2276):20230422, 2023.

\bibitem[BG{\'S}20]{BGS20}
K.~Bara{\'n}ski, Y.~Gutman, and A.~{\'S}piewak.
\newblock A probabilistic {T}akens theorem.
\newblock {\em Nonlinearity}, 33(9):4940--4966, 2020.

\bibitem[BG{\'S}22a]{BGS22}
K.~Bara{\'n}ski, Y.~Gutman, and A.~{\'S}piewak.
\newblock On the {S}hroer-{S}auer-{O}tt-{Y}orke predictability conjecture for
  time-delay embeddings.
\newblock {\em Commun. Math. Phys.}, 391:609--641, 2022.

\bibitem[BG{\'S}22b]{baranski2022prediction}
K.~Bara{\'n}ski, Y.~Gutman, and A.~{\'S}piewak.
\newblock Prediction of dynamical systems from time-delayed measurements with
  self-intersections.
\newblock {\em arXiv preprint arXiv:2212.13509}, 2022.

\bibitem[BK86]{broomhead1986extracting}
D.~S. Broomhead and G.~P. King.
\newblock Extracting qualitative dynamics from experimental data.
\newblock {\em Physica D: Nonlinear Phenomena}, 20(2-3):217--236, 1986.

\bibitem[Bog07]{Boga}
V.~I. Bogachev.
\newblock {\em Measure theory. {V}ol. {I}, {II}}.
\newblock Springer-Verlag, Berlin, 2007.

\bibitem[Br{\"o}09]{Bröcker}
J.~Br{\"o}cker.
\newblock Reliability, sufficiency, and the decomposition of proper scores.
\newblock {\em Q.~J.~R. Meteorol. Soc.}, 135:1512--1519, 2009.

\bibitem[BSS{\etalchar{+}}15]{RatErg}
F.~Bozgan, A.~Sanchez, C.~E. Silva, D.~Stevens, and J.~Wang.
\newblock Subsequence bounded rational ergodicity of rank-one transformations.
\newblock {\em Dynamical Systems}, 30:70--84, 2015.

\bibitem[CAS23]{colbrook2023residual}
M.~J. Colbrook, L.~J. Ayton, and M.~Sz{\H{o}}ke.
\newblock Residual dynamic mode decomposition: robust and verified
  {K}oopmanism.
\newblock {\em Journal of Fluid Mechanics}, 955:A21, 2023.

\bibitem[CKM02]{Cai}
D.~Cai, R.~Kleeman, and A.~Majda.
\newblock A mathematical framework for quantifying predictability through
  relative entropy.
\newblock {\em Methods Appl. Anal.}, 9(3):425--444, 2002.

\bibitem[Col23]{colbrook2023mpedmd}
M.~J. Colbrook.
\newblock The {mpEDMD} algorithm for data-driven computations of
  measure-preserving dynamical systems.
\newblock {\em SIAM Journal on Numerical Analysis}, 61(3):1585--1608, 2023.

\bibitem[CT24]{colbrook2021rigorous}
M.~J. Colbrook and A.~Townsend.
\newblock Rigorous data-driven computation of spectral properties of {K}oopman
  operators for dynamical systems.
\newblock {\em Communications on Pure and Applied Mathematics}, 77(1):221--283,
  2024.

\bibitem[DG19]{GiDa19}
S.~Das and D.~Giannakis.
\newblock Delay-coordinate maps and the spectra of {K}oopman operators.
\newblock {\em Journal of Statistical Physics}, 175(6):1107--1145, 2019.

\bibitem[DHVMZ16]{dellnitz2016computation}
M.~Dellnitz, M.~Hessel-Von~Molo, and A.~Ziessler.
\newblock On the computation of attractors for delay differential equations.
\newblock {\em Journal of Computational Dynamics}, 3(1):93--112, 2016.

\bibitem[DJ76]{Junco_1976}
A.~Del~Junco.
\newblock Transformations with discrete spectrum are stacking transformations.
\newblock {\em Canadian Journal of Mathematics}, 28(4):836–839, 1976.

\bibitem[Dun77]{Dun77}
R.~Duncan.
\newblock Some pointwise convergence results in ${L}^p(\mu)$, $1<p<\infty$.
\newblock {\em Canad. Math. Bull}, 20:3, 1977.

\bibitem[EFHN15]{EFHN15}
T.~Eisner, B.~Farkas, M.~Haase, and R.~Nagel.
\newblock {\em Operator theoretic aspects of ergodic theory}, volume 272.
\newblock Springer, 2015.

\bibitem[ENB00]{EN00}
K.-J. Engel, R.~Nagel, and S.~Brendle.
\newblock {\em One-parameter semigroups for linear evolution equations}, volume
  194.
\newblock Springer, 2000.

\bibitem[Fer97]{ferenczi1997systems}
S.~Ferenczi.
\newblock Systems of finite rank.
\newblock In {\em Colloquium Mathematicae}, volume~73, pages 35--65, 1997.

\bibitem[FS87]{FS87}
J.~D. Farmer and J.~J. Sidorowich.
\newblock Predicting chaotic time series.
\newblock {\em Phys. Rev. Lett.}, 59:845--848, Aug 1987.

\bibitem[GGH21]{gilani2021lernel}
F.~Gilani, D.~Giannakis, and J.~Harlim.
\newblock Kernel-based prediction of non-{M}arkovian time series.
\newblock {\em Physica D: Nonlinear Phenomena}, 418:132829, 2021.

\bibitem[GHL20]{geurts2020lyapunov}
B.~J. Geurts, D.~D. Holm, and E.~Luesink.
\newblock Lyapunov exponents of two stochastic {L}orenz 63 systems.
\newblock {\em Journal of Statistical Physics}, 179:1343--1365, 2020.

\bibitem[Gia21]{giannakis2021delay}
D.~Giannakis.
\newblock Delay-coordinate maps, coherence, and approximate spectra of
  evolution operators.
\newblock {\em Research in the Mathematical Sciences}, 8(1):8, 2021.

\bibitem[GPD17]{gouasmi2017apriori}
A.~Gouasmi, E.~J. Parish, and K.~Duraisamy.
\newblock \textit{A priori} estimation of memory effects in reduced-order
  models of nonlinear systems using the {M}ori--{Z}wanzig formalism.
\newblock {\em Proceedings of the Royal Society A: Mathematical, Physical and
  Engineering Sciences}, 473(2205):20170385, 2017.

\bibitem[GR07]{Gneiting}
T.~Gneiting and A.~Raftery.
\newblock Strictly proper scoring rules, prediction, and estimation.
\newblock {\em Journal of the American Statistical Association},
  102(477):359--378, 2007.

\bibitem[GR21a]{gottwald2021combining}
G.~A. Gottwald and S.~Reich.
\newblock Combining machine learning and data assimilation to forecast
  dynamical systems from noisy partial observations.
\newblock {\em Chaos: An Interdisciplinary Journal of Nonlinear Science},
  31(10):101103, 2021.

\bibitem[GR21b]{gottwald2021supervised}
G.~A. Gottwald and S.~Reich.
\newblock Supervised learning from noisy observations: Combining
  machine-learning techniques with data assimilation.
\newblock {\em Physica D: Nonlinear Phenomena}, 423:132911, 2021.

\bibitem[Gut16]{Gutman}
Y.~Gutman.
\newblock Takens’ embedding theorem with a continuous observable.
\newblock In {\em Ergodic Theory: Advances in Dynamical Systems}, pages
  134--141. De Gruyter, Berlin-Boston, 2016.

\bibitem[Hal17]{Hal17}
P.~R. Halmos.
\newblock {\em Lectures on ergodic theory}.
\newblock Courier Dover Publications, 2017.

\bibitem[HBS15]{HBS}
F.~Hamilton, T.~Berry, and T.~Sauer.
\newblock Predicting chaotic time series with a partial model.
\newblock {\em Phys. Rev. E}, 92:010902, Jul 2015.

\bibitem[HGLS05]{HsiehEtAl}
C.-H. Hsieh, S.~Glaser, A.~Lucas, and G.~Sugihara.
\newblock Distinguishing random environmental fluctuations from ecological
  catastrophes for the {N}orth {P}acific {O}cean.
\newblock {\em Nature}, 435:336--340, 2005.

\bibitem[Huk06]{huke2006embedding}
J.~P. Huke.
\newblock Embedding nonlinear dynamical systems: A guide to takens' theorem.
\newblock 2006.

\bibitem[KKBK20]{kamb2020time}
M.~Kamb, E.~Kaiser, S.~L. Brunton, and J.~N. Kutz.
\newblock Time-delay observables for {K}oopman: Theory and applications.
\newblock {\em SIAM Journal on Applied Dynamical Systems}, 19(2):886--917,
  2020.

\bibitem[KM18]{korda2018convergence}
M.~Korda and I.~Mezi{\'c}.
\newblock On convergence of extended dynamic mode decomposition to the
  {K}oopman operator.
\newblock {\em Journal of Nonlinear Science}, 28:687--710, 2018.

\bibitem[Kol41a]{kolmogorov1941interpolation}
A.~Kolmogorov.
\newblock Interpolation and extrapolation of stationary random sequences.
\newblock {\em Izvestiya Rossiiskoi Akademii Nauk. Seriya Matematicheskaya},
  5:3, 1941.

\bibitem[Kol41b]{kolmogorov1941stationary}
A.~Kolmogorov.
\newblock Stationary sequences in {H}ilbert space.
\newblock {\em Bull. Math. Univ., Moscou}, 2, 1941.

\bibitem[Koo31]{koopman1931hamiltonian}
B.~O. Koopman.
\newblock Hamiltonian systems and transformation in {H}ilbert space.
\newblock {\em Proceedings of the National Academy of Sciences},
  17(5):315--318, 1931.

\bibitem[KPM20]{korda2020data}
M.~Korda, M.~Putinar, and I.~Mezi{\'c}.
\newblock Data-driven spectral analysis of the {K}oopman operator.
\newblock {\em Applied and Computational Harmonic Analysis}, 48(2):599--629,
  2020.

\bibitem[Kre91]{kreyszig1991introductory}
E.~Kreyszig.
\newblock {\em Introductory functional analysis with applications}, volume~17.
\newblock John Wiley \& Sons, 1991.

\bibitem[KS67]{KatokStepin}
A.~Katok and A.~Stepin.
\newblock Approximations in ergodic theory.
\newblock {\em Uspehi Mat. Nauk}, pages 81--106, 1967.

\bibitem[KT06]{KatokThouvenot}
A.~Katok and J.-P. Thouvenot.
\newblock Spectral properties and combinatorial constructions in ergodic
  theory.
\newblock In {\em Handbook of dynamical systems}, volume~1B, pages 649--743,
  2006.

\bibitem[KvN32]{koopman1932dynamical}
B.~O. Koopman and J.~von Neumann.
\newblock Dynamical systems of continuous spectra.
\newblock {\em Proceedings of the National Academy of Sciences},
  18(3):255--263, 1932.

\bibitem[KY90]{KostelichYorke}
E.~Kostelich and J.~Yorke.
\newblock Noise reduction: Finding the simplest dynamical system consistent
  with the data.
\newblock {\em Physica D: Nonlinear Phenomena}, 41(2):183--196, 1990.

\bibitem[LF87]{lapedes1987nonlinear}
A.~Lapedes and R.~Farber.
\newblock Nonlinear signal processing using neural networks: Prediction and
  system modelling.
\newblock Technical report, 1987.

\bibitem[LL21]{lin2021data}
K.~K. Lin and F.~Lu.
\newblock Data-driven model reduction, {W}iener projections, and the
  {K}oopman--{M}ori--{Z}wanzig formalism.
\newblock {\em Journal of Computational Physics}, 424:109864, 2021.

\bibitem[LMP05]{luzzatto2005lorenz}
S.~Luzzatto, I.~Melbourne, and F.~Paccaut.
\newblock The {L}orenz attractor is mixing.
\newblock {\em Communications in Mathematical Physics}, 260:393--401, 2005.

\bibitem[LTPL23]{lin23regression}
Y.~T. Lin, Y.~Tian, D.~Perez, and D.~Livescu.
\newblock Regression-based projection for learning {M}ori–-{Z}wanzig
  operators.
\newblock {\em SIAM Journal on Applied Dynamical Systems}, 22(4):2890--2926,
  2023.

\bibitem[Mez22]{mezic2022numerical}
I.~Mezi{\'c}.
\newblock On numerical approximations of the {K}oopman operator.
\newblock {\em Mathematics}, 10(7):1180, 2022.

\bibitem[MG07]{MeGo07}
I.~Melbourne and G.~A. Gottwald.
\newblock Power spectra for deterministic chaotic dynamical systems.
\newblock {\em Nonlinearity}, 21(1):179, 2007.

\bibitem[MGNS18]{Fish}
S.~Munch, A.~Giron-Nava, and G.~Sugihara.
\newblock Nonlinear dynamics and noise in fisheries recruitment: a global
  meta-analysis.
\newblock {\em Fish and Fisheries}, 19:964--973, 2018.

\bibitem[Mor65]{Mori65}
H.~Mori.
\newblock Transport, collective motion, and {B}rownian motion.
\newblock {\em Progress of Theoretical Physics}, 33(3):423--455, 03 1965.

\bibitem[Nad20]{Nad20}
M.~Nadkarni.
\newblock {\em Spectral Theory of Dynamical Systems}.
\newblock Springer Singapore, 2nd edition, 2020.

\bibitem[PCFS80]{packard1980geometry}
N.~H. Packard, J.~P. Crutchfield, J.~D. Farmer, and R.~S. Shaw.
\newblock Geometry from a time series.
\newblock {\em Physical Review Letters}, 45(9):712, 1980.

\bibitem[Phi17]{philipp2017bessel}
F.~Philipp.
\newblock {B}essel orbits of normal operators.
\newblock {\em J.\ Math.\ Anal.\ Appl.}, 448(2):767--785, 2017.

\bibitem[PW93]{percival1993spectral}
D.~B. Percival and A.~T. Walden.
\newblock {\em Spectral analysis for physical applications}.
\newblock Cambridge University Press, 1993.

\bibitem[RMB{\etalchar{+}}09]{rowley2009spectral}
C.~W. Rowley, I.~Mezi{\'c}, S.~Bagheri, P.~Schlatter, and D.~S. Henningson.
\newblock Spectral analysis of nonlinear flows.
\newblock {\em Journal of fluid mechanics}, 641:115--127, 2009.

\bibitem[Rob05]{Robinson}
J.~C. Robinson.
\newblock A topological delay embedding theorem for infinite-dimensional
  dynamical systems.
\newblock {\em Nonlinearity}, 18(5):2135--2143, 2005.

\bibitem[Rob10]{robinson_Book}
J.~C. Robinson.
\newblock {\em Dimensions, Embeddings, and Attractors}.
\newblock Cambridge Tracts in Mathematics. Cambridge University Press, 2010.

\bibitem[Rud87]{Rud87}
W.~Rudin.
\newblock {\em Real and Complex Analysis}.
\newblock McGraw-Hill Book Co., 3rd edition, 1987.

\bibitem[SGM{\etalchar{+}}90]{SugiharaEtAl}
G.~Sugihara, B.~Grenfell, R.~May, P.~Chesson, H.~M. Platt, and M.~Williamson.
\newblock Distinguishing error from chaos in ecological time series.
\newblock {\em Philosophical Transactions: Biological Sciences},
  330(1257):235--251, 1990.

\bibitem[Sim10]{simon2010szegHo}
B.~Simon.
\newblock {\em Szeg{\H{o}}'s theorem and its descendants}.
\newblock Princeton University Press, 2010.

\bibitem[SM90]{SugiharaMay}
G.~Sugihara and R.~May.
\newblock Nonlinear forecasting as a way of distinguishing chaos from
  measurement error in time series.
\newblock {\em Nature}, 344:734--741, 1990.

\bibitem[SS08]{SS08}
P.~J. Schmid and J.~Sesterhenn.
\newblock Dynamic mode decomposition of numerical and experimental data.
\newblock In {\em Bull. Amer. Phys. Soc.}, volume~61, 2008.

\bibitem[SYC91]{SaYoCa91}
T.~Sauer, J.~A. Yorke, and M.~Casdagli.
\newblock Embedology.
\newblock {\em Journal of Statistical Physics}, 65(3):579--616, 1991.

\bibitem[Tak81]{Takens}
F.~Takens.
\newblock Detecting strange attractors in turbulence.
\newblock In {\em Dynamical Systems and Turbulence, Warwick 1980, volume 898 of
  Lecture Notes in Math.}, pages 366--381. Springer, Berlin-New York, 1981.

\bibitem[Tak02]{Takens2}
F.~Takens.
\newblock The reconstruction theorem for endomorphisms.
\newblock {\em Bull. Braz. Math. Soc. (N.S.)}, 33(2):231--262, 2002.

\bibitem[Tay13]{taylor2013partial}
M.~Taylor.
\newblock {\em Partial differential equations {II}: Qualitative studies of
  linear equations}, volume 116.
\newblock Springer Science \& Business Media, 2013.

\bibitem[TE05]{TrEm05}
L.~N. Trefethen and M.~Embree.
\newblock {\em Spectra and pseudospectra: the behavior of nonnormal matrices
  and operators}.
\newblock Princeton University Press, 2005.

\bibitem[VG23]{valva2023consistent}
C.~Valva and D.~Giannakis.
\newblock Consistent spectral approximation of {K}oopman operators using
  resolvent compactification.
\newblock {\em arXiv preprint arXiv:2309.00732}, 2023.

\bibitem[vN32]{neumann1932operatorenmethode}
J.~von Neumann.
\newblock Zur {O}peratorenmethode in der klassischen {M}echanik.
\newblock {\em Annals of Mathematics}, pages 587--642, 1932.

\bibitem[Vos03]{Voss}
H.~U. Voss.
\newblock Synchronization of reconstructed dynamical systems.
\newblock {\em Chaos}, 13:327--334, 2003.

\bibitem[Wal00]{Wal00}
P.~Walters.
\newblock {\em An introduction to ergodic theory}, volume~79.
\newblock Springer Science \& Business Media, 2000.

\bibitem[WG22]{WangGuet2022}
Z.~Wang and C.~Guet.
\newblock Self-consistent learning of neural dynamical systems from noisy time
  series.
\newblock {\em IEEE Transactions on Emerging Topics in Computational
  Intelligence}, 6(5):1103--1112, 2022.

\bibitem[Wie49]{wiener1949extrapolation}
N.~Wiener.
\newblock {\em Extrapolation, interpolation, and smoothing of stationary time
  series}, volume 113.
\newblock The MIT Press, 1949.

\bibitem[WKR15]{williams2015data}
M.~O. Williams, I.~G. Kevrekidis, and C.~W. Rowley.
\newblock A data--driven approximation of the {K}oopman operator: {E}xtending
  dynamic mode decomposition.
\newblock {\em Journal of Nonlinear Science}, 25:1307--1346, 2015.

\bibitem[WKSS21]{wulkow2021data}
N.~Wulkow, P.~Koltai, V.~Sunkara, and C.~Sch{\"u}tte.
\newblock Data-driven modelling of nonlinear dynamics by barycentric
  coordinates and memory.
\newblock {\em arXiv preprint arXiv:2112.06742}, 2021.

\bibitem[YG23]{young2023deep}
C.~D. Young and M.~D. Graham.
\newblock Deep learning delay coordinate dynamics for chaotic attractors from
  partial observable data.
\newblock {\em Physical Review E}, 107(3):034215, 2023.

\bibitem[Zwa61]{zwanzig1961statistical}
R.~Zwanzig.
\newblock Statistical mechanics of irreversibility.
\newblock In W.~Brittin, editor, {\em Lectures in Theoretical Physiscs},
  volume~3. Wiley-Interscience, New York, NY, USA, 1961.

\end{thebibliography}
	}
	
	\appendix
	
	\section{Proof of Lemma~\ref{lem:weakapprox}}
	\label{app:aux_proofs}
	
	\begin{proof}[Proof of Lemma~\ref{lem:weakapprox}.]
		The statements (1) $\Leftrightarrow$ (3) $\Leftrightarrow$ (4) can also be found in~\cite[p.~62]{Hal17}.\\
		
		(1) $\Leftrightarrow$ (2). Let $A \in\mathfrak{B}$. Since $T$ preserves $\mu$, we have $\mu(T_n A \, \triangle \, T A) = \mu(T^{-1}T_n A \, \triangle \, A)$. Let $B:=T_n A$. We obtain $\mu(T_n A \, \triangle \, T A) = \mu(T^{-1}B \, \triangle \, T_n^{-1}B)$. As $T,T_n$ are measure-preserving essential bijections, $B\in \mathfrak{B}$, and in fact $T \mathfrak{B} = \mathfrak{B} = T_n \mathfrak{B}$ for every~$n$. We thus obtain the equivalence of $T_n \stackrel{w}{\to} T$ and $T_n^{-1} \stackrel{w}{\to} T^{-1}$ as~$n\to\infty$.
		
		(1) $\Rightarrow$ (5). We will show the claim for $i\ge 0$, while $i<0$ then follows with (1) $\Leftrightarrow$~(2). There is nothing to show for $i=0$, let us thus consider $i=1$ first. For arbitrary $\ep>0$ let $g_\ep = \sum_{j=1}^{N_\ep} \alpha_j \mathds{1}_{A_j}$, $A_j\in\mathfrak{B}$, be a simple function with $\| g - g_\ep \|_{L^2_\mu} < \ep$. We note that 
		\begin{equation*}
			\label{eq:simplediff}
			\| g_\ep\circ T - g_\ep \circ T_n \|_{L^2_\mu} \le \sum_{j=1}^{N_\ep} |\alpha_j| \|\mathds{1}_{A_j}\circ T - \mathds{1}_{A_j}\circ T_n\|_{L^2_\mu} =  \sum_{j=1}^{N_\ep} |\alpha_j| \mu(T_n A_j \, \triangle \, T A_j)^{1/2}.
		\end{equation*}
		By assumption, we thus have $\| g_\ep\circ T - g_\ep \circ T_n \|_{L^2_\mu} \to 0$ as $n\to\infty$. Using this and the measure-preserving property of $T,T_n$, we obtain for sufficiently large $n$ that
		\begin{align*}
			\| g\circ T - g\circ T_n \|_{L^2_\mu} &\le \|g\circ T - g_\ep\circ T\|_{L^2_\mu} + \| g_\ep\circ T - g_\ep\circ T_n \|_{L^2_\mu} + \| g_\ep\circ T_n - g\circ T_n\|_{L^2_\mu}\\
			&< \ep+\ep+\ep.
		\end{align*}
		Thus, $g\circ T_n\to g\circ T$ in $L^2_\mu$ as $n\to\infty$. The general statement follows by simple induction, as
		\[
		\| g\circ T^{i+1} - g\circ T_n^{i+1}\|_{L^2_\mu} \le \| g\circ T^i \circ T -g\circ T^i \circ T_n\|_{L^2_\mu} + \| g\circ T^i \circ T_n - g\circ T_n^i \circ T_n \|_{L^2_\mu},
		\]
		and $T_n$ is measure-preserving.
		
		(5) $\Rightarrow$ (3). Trivial by setting~$i=1$.
		
		(3) $\Rightarrow$ (1). Trivial by taking $g$ as characteristic functions.
		
		(3) $\Leftrightarrow$ (4). The first implication is trivial. The second is a straightforward computation, using that $\cU_T,\cU_{T_n}$ are unitary operators.
		
	\end{proof}
	
\end{document}